\documentclass[reqno,12pt]{amsart}
\theoremstyle{plain}
\newtheorem{thm}{Theorem}[section]
\newtheorem{cor}[thm]{Corollary} 
\newtheorem{lemma}[thm]{Lemma} 
\newtheorem{prop}[thm]{Proposition}

\theoremstyle{remark}

\newtheorem{remark}[thm]{Remark}

\theoremstyle{definition}

\newtheorem{defi}[thm]{Definition}
\newtheorem{example}[thm]{Example}

\usepackage{amscd,amssymb,comment,epic,eepic,euscript,graphics}

\newcount\theTime
\newcount\theHour
\newcount\theMinute
\newcount\theMinuteTens
\newcount\theScratch
\theTime=\number\time
\theHour=\theTime
\divide\theHour by 60
\theScratch=\theHour
\multiply\theScratch by 60
\theMinute=\theTime
\advance\theMinute by -\theScratch
\theMinuteTens=\theMinute
\divide\theMinuteTens by 10
\theScratch=\theMinuteTens
\multiply\theScratch by 10
\advance\theMinute by -\theScratch

\def\today{{\number\day\space
 \ifcase\month\or
  January\or February\or March\or April\or May\or June\or
  July\or August\or September\or October\or November\or December\fi
 \space\number\year}}

\newcommand{\la}{\lambda}
\newcommand{\De}{\Delta}

\newcommand\Rfr{{\mathfrak R}}
\newcommand\Mcal{{\mathcal{M}}} 
\newcommand\Reals{{\mathbf R}}
\newcommand\Complex{{\mathbf C}}
\newcommand\Nats{{\mathbf N}}
\newcommand{\BH}{\mathcal{B}(\mathcal{H})}
\newcommand{\Hcal}{\mathcal{H}}
\newcommand\Tr{{\mathrm{Tr}}}
\newcommand{\st}{\,:\,}
\newcommand{\im}{\text{\rm Im}}
\newcommand{\re}{\text{\rm Re}}
\renewcommand{\i}{\text{\rm i}}

\newcommand{\norm}[1]{\left\Vert#1\right\Vert}
\newcommand\smdb[2]{\underset{#2}{\underbrace{#1}}}
\newcommand\ncl[1]{{L_#1(\Mcal,\tau)}}
\newcommand\ncs[1]{{\mathcal{L}_#1(\Mcal,\tau)}}
\newcommand{\dd}[3]{\De_{\la_1,\dots,#1}^{(#2)}\left(#3\right)}

\begin{document}

\title[Spectral shift]{Higher order spectral shift}

\author[Dykema]{Ken Dykema$^{*}$}
\author[Skripka]{Anna Skripka}
\address{Department of Mathematics, Texas A\&M University,
College Station, TX 77843-3368, USA}
\email{kdykema@math.tamu.edu}
\email{askripka@math.tamu.edu}
\thanks{\footnotesize $^{*}$Research supported in part by NSF grant DMS-0600814}



\date{17 December, 2008}

\begin{abstract}
We construct higher order spectral shift functions, extending the perturbation theory results of M. G. Krein \cite{Krein1} and L. S. Koplienko \cite{Kop84} on representations for the remainders of the first and second order Taylor-type approximations of operator functions. The higher order spectral shift functions represent the remainders of higher order Taylor-type approximations; they can be expressed recursively via the lower order (in particular, Krein's and Koplienko's) ones. We also obtain higher order spectral averaging formulas generalizing the Birman-Solomyak spectral averaging formula. The results are obtained in the semi-finite von Neumann algebra setting, with the perturbation taken in the Hilbert-Schmidt class of the algebra.
\end{abstract}

\maketitle

\section{Introduction}

Let $\Hcal$ be a separable Hilbert space and $\BH$ the algebra of bounded linear operators on $\Hcal$. Let $\Mcal$ be a semi-finite von Neumann algebra acting on $\Hcal$ and $\tau$ a semi-finite normal faithful trace on $\Mcal$.
We study how the value $f(H_0)$ of a function $f$ on a self-adjoint operator $H_0$ in $\Mcal$ changes under a perturbation $V=V^*\in\Mcal$ of the operator argument $H_0$.
It is well known that for certain functions $f$,
the value $f(H_0+V)$ can be approximated by the Fr\'{e}chet derivatives of the mapping $H^*=H\mapsto f(H)$ at point $H_0$.

\begin{thm}$($cf. \cite[Theorem 1.43, Corollary 1.45]{Schwartz}$)$\label{prop:Rpf}
Let $f:\Reals\to\Complex$ be a bounded function such that the mapping $H\mapsto f(H)$ defined on self-adjoint elements of $\BH$ is $p$ times continuously differentiable in the sense of Fr\'{e}chet (and, hence, in the sense of G\^ateaux). Let $H_0=H_0^*,V=V^*\in\BH$ and denote \begin{align}\label{f-la:rem}
R_{p,H_0,V}(f)=f(H_0+V)-\sum_{j=0}^{p-1}\frac{1}{j!}\frac{d^j}{dt^j}
\bigg|_{t=0}f(H_0+tV).\end{align}
Then \begin{align}\label{f-la:Rpf} R_{p,H_0,V}(f)=\frac{1}{(p-1)!}\int_0^1
(1-t)^{p-1}\frac{d^p}{dt^p}f(H_0+tV)\,dt
\end{align} and
\begin{align}\label{f-la:as}
\norm{R_{p,H_0,V}(f)}=\mathcal{O}(\norm{V}^p).
\end{align}
\end{thm}
Theorem \ref{prop:Rpf} generalizes the Taylor approximation theorem for scalar functions.
It was proved in \cite{Daletskii} that for $f\in C^{2p}(\Reals)$, the operator function $f$ is Fr\'{e}chet differentiable $p$ times on $\BH$, with the derivative written as an iterated operator integral. For $f\in\mathcal{W}_p$ (the set of functions $f\in
C^p(\Reals)$ such that for each $j=0,\dots,p$, the derivative $f^{(j)}$ equals the Fourier transform $\int_\Reals e^{\i t\la}\,d\mu_{f^{(j)}}(\la)$ of a finite Borel
measure $\mu_{f^{(j)}}$) and a (possibly) unbounded $H_0$, the differentiability of $H\mapsto f(H)$ in the sense of Fr\'{e}chet of order $p$ was established in \cite{Azamov0}; in that case, the G\^ateaux derivative $\frac{d^p}{dt^p}f(H_0+tV)$ was represented as a Bochner-type multiple operator integral. For $f$ in the Besov class $B^1_{\infty1}(\Reals)\cap B^p_{\infty1}(\Reals)$, it is known that the G\^ateaux derivative of $f$ of order $p$ exists \cite{PellerMult}, but the bound \eqref{f-la:as} has not been proved.

In the scalar case ($\dim(\Hcal)=1$), we have that $\tau[R_{p,H_0,V}(f)]$ is a bounded functional on the space of functions $f^{(p)}$ and
\begin{align}\label{f-la:dh1}
|\tau[R_{p,H_0,V}(f)]|\leq\frac{\tau(|V|^p)}{p!}\norm{f^{(p)}}_\infty.
\end{align} In the case of a nontrivial $\Hcal$ ($\dim(\Hcal)>1$), it is generally hard to separate contribution of the perturbation $V$ to the estimate for the remainder \eqref{f-la:as} from contribution of the scalar function $f^{(p)}$. One of approaches to \eqref{f-la:dh1} is the estimate $\norm{R_{p,H_0,V}(f)}\leq C(H_0,V)\norm{f^{(2p)}}_\infty$, for $f\in C^{2p}(\Reals)$ \cite{Daletskii}, with $C(H_0,V)$ a constant depending on bounded self-adjoint operators $H_0$ and $V$.
Another approach is the estimate
\begin{align}\label{f-la:mest}
|\tau[R_{p,H_0,V}(f)]|\leq\frac{\tau(|V|^p)}{p!}\norm{\mu_{f^{(p)}}}
\end{align} \cite{Dostanic} (see \cite{GPS} for an example when $\norm{\mu_{f^{(p)}}}$ can be replaced with $\norm{\widehat{f^{(p)}}}_1$), for $\tau$ the usual trace, $H_0=H_0^*$ an operator in $\Hcal$, $V=V^*$ an operator in the Schatten $p$-class, and $f\in\mathcal{W}_p$.
If $H_0=H_0^*$ is affiliated with a semi-finite von Neumann algebra $\Mcal$, $V=V^*$ is in the $\tau$-Schatten $p$-class of $\Mcal$, and $f\in\mathcal{W}_p$, then the remainder $R_{p,H_0,V}(f)$ belongs to the Schatten $p$-class of $\Mcal$ as well and
\begin{align*}
\left[\tau(|R_{p,H_0,V}(f)|^p)\right]^{1/p}+\norm{R_{p,H_0,V}(f)}\leq
\frac{\left([\tau(|V|^p)]^{1/p}+\norm{V}\right)^p}{p!}\norm{\mu_{f^{(p)}}}
\end{align*} see \cite{Azamov0}. 

In the particular case of $p=1$ or $p=2$, the functional $\tau[R_{p,H_0,V}(f)]$ is bounded on the space of functions $f'$ or $f''$, respectively, and \eqref{f-la:dh1} holds. The measure representing the functional is absolutely continuous (with respect to Lebesgue's measure), with the density equal to Krein's spectral shift function $\xi_{H_0+V,H_0}$ or Koplienko's spectral shift function $\eta_{H_0,H_0+V}$, respectively. That is, we have
\begin{align}\label{f-la:trp=1}
\tau[R_{1,H_0,V}(f)]=\int_\Reals f'(t)\xi_{H_0+V,H_0}(t)\,dt,
\quad\quad|\tau[R_{1,H_0,V}(f)]|\leq\tau(|V|)\norm{f'}_\infty
\end{align} and
\begin{align}\label{f-la:trp=2}
\tau[R_{2,H_0,V}(f)]=\int_\Reals f''(t)\eta_{H_0,H_0+V}(t)\,dt,
\quad\quad|\tau[R_{2,H_0,V}(f)]|\leq\frac{\tau(|V|^2)}{2}\norm{f''}_\infty.
\end{align}

Existence of $\xi_{H_0+V,H_0}$, with $\tau(|V|)<\infty$, satisfying \eqref{f-la:trp=1} for $f\in\mathcal{W}_1$, was proved in the setting $\Mcal=\BH$ in \cite{Krein1} (cf. also \cite{Lifshits})
and extended to the setting of an arbitrary semi-finite von Neumann algebra $\Mcal$ in \cite{Azamov,Carey}. Moreover, when $\Mcal=\BH$, the trace formula in \eqref{f-la:trp=1} is known to hold for  $f\in B^1_{\infty 1}(\Reals)$ \cite{PellerKr}. In the setting $\Mcal=\BH$, existence of $\eta_{H_0,H_0+V}$, with $V$ in the Hilbert-Schmidt class, satisfying \eqref{f-la:trp=2} for bounded rational functions $f$ was proved in \cite{Kop84}. Later, it was proved in \cite{PellerKo} that $\eta_{H_0,H_0+V}$ satisfies the trace formula in \eqref{f-la:trp=2} for functions $f$ in $B^1_{\infty1}(\Reals)\cap B^p_{\infty1}(\Reals)$.
When $V$ is in the trace class, Koplienko's spectral shift function can be written explicitly as \begin{align}\label{f-la:koviaxi}
\eta_{H_0,H_0+V}(t)=-\int_{-\infty}^t\xi_{H_0+V,H_0}(\la)\,d\la+\tau[E_{H_0}((-\infty,t))V],
\end{align} where $E_{H_0}$ is the spectral measure of $H_0$ \cite{Kop84}. In the context of a general $\Mcal$, Koplienko's spectral shift function $\eta_{H_0,H_0+V}$, with $\tau(|V|)<\infty$, and the representation \eqref{f-la:koviaxi} are discussed in \cite{convexity}.

For $p\geq 3$, $\Mcal=\BH$, and $\tau(|V|^p)<\infty$, the distribution $\tau[R_{p,H_0,V}(f)]$ is given by an $L^2$-function $\gamma_{p,H_0,V}$ satisfying
\[\tau[R_{p,H_0,V}(f)]=\frac{\tau(V^p)}{p!}f^{(p)}(0)+\int_\Reals f^{(p+1)}(t)\gamma_{p,H_0,V}(t)\,dt,\] for all $f\in\mathcal{W}_{p+1}$ \cite{Dostanic}.
It was conjectured in \cite{Kop84} that there exists a Borel measure $\nu_p$ with the total variation bounded by $\frac{\tau(|V|^p)}{p!}$ such that
\begin{align}\label{f-la:tr_intro}
\tau[R_{p,H_0,V}(f)]=\int_\Reals f^{(p)}(t)\,d\nu_p(t),\end{align}
for bounded rational functions $f$. Unfortunately, the proof of \eqref{f-la:tr_intro} in \cite{Kop84} was based on the false claim that for $V$ in the Schatten $p$-class, $p>2$, the set function defined on rectangles of $\Reals^{p+1}$ by
\begin{align}
\label{f-la:sf}
A_1\times A_2\times\dots\times A_{p+1}\mapsto\tau[E(A_1)VE(A_2)V\dots VE(A_{p+1})],
\end{align} where $E(\cdot)$ is a spectral measure on $\Reals$ with values in $\BH$, extends to a (countably additive) measure of bounded variation (see a counterexample in section \ref{sec:mm}). When $V$ is in the Hilbert-Schmidt class of $\Mcal=\BH$, the set function in \eqref{f-la:sf} does extend to a (countably-additive) measure of bounded variation \cite{Birman96,Pavlov} and thus ideas of \cite{Kop84} can be applied to prove existence of a measure $\nu_p$ satisfying \eqref{f-la:tr_intro} for bounded rational functions (see section \ref{sec:bh}). In this case, the total variation of $\nu_p$ is bounded by \[\norm{\nu_p}\leq\frac{\big(\tau\left(|V|^2\right)\big)^{p/2}}{p!}.\] Adjusting techniques of \cite{Dostanic} then extends \eqref{f-la:tr_intro} to the functions $f\in\mathcal{W}_p$.

For $\Mcal$ a von Neumann algebra acting on an infinite-dimensional Hilbert space $\Hcal$, the set function in \eqref{f-la:sf},
with $E(\cdot)$ the spectral measure attaining its values in $\Mcal$ and $V\in\Mcal$ satisfying $\tau(|V|^2)<\infty$, may fail to extend to a finite measure on $\Reals^{p+1}$ for $p>2$ even if $\tau$ is finite (see a counterexample in section \ref{sec:mm}). Therefore, the approach of \cite{Kop84} is not applicable in the proof of \eqref{f-la:tr_intro}. When $\Mcal$ is a general semi-finite von Neumann algebra, we prove \eqref{f-la:tr_intro} for $p=3$ by relating $R_{3,H_0,V}$ to $R_{2,H_0,V}$, which allows to reduce the problem to the case of $p=2$ (see sections \ref{sec:ct} and \ref{sec:vn}). We also study the case when $\Mcal$ is finite and $H_0,V\in\Mcal$ are free with respect to the finite trace $\tau$ (which is assumed
normalized so that $\tau(1)=1$).
Freeness was introduced by Voiculescu (see, for example, \cite{VDN})
and amounts to a specific prescription for the values of
the mixed moments of $H_0$ and $V$ in terms of the individual moments
of $H_0$ and $V$.
Free perturbations have appeared in the study of
quite general
operators in finite von Neumann algebras,
for example in the seminal work of Haagerup and
Schultz~\cite{HS}.
Assuming freeness, we show that for all $p$ the set function in \eqref{f-la:sf} extends
to a finite measure on $\Reals^{p+1}$ (see section \ref{sec:mm}), from which \eqref{f-la:tr_intro} can be derived.

Under the assumptions that we impose to prove existence of $\nu_p$ satisfying \eqref{f-la:tr_intro} (see discussion in the two preceding paragraphs), we also construct a function $\eta_p$, called the $p$th-order spectral shift function, such that $d\nu_p(t)=\eta_p(t)\,dt$, provided $H_0$ is bounded (see statements in section \ref{sec:mr} and proofs in sections \ref{sec:bh} and \ref{sec:vn}). The spectral shift function of order $p$ admits the recursive representation
\begin{align}
\label{f-la:eta_intro}
\eta_p(t)=-\int_{-\infty}^t\eta_{p-1}(\la)\,d\la+\int_{\Reals^{p-1}}
spline_{\la_1,\dots,\la_{p-1}}(t)\,dm_{p-1,H_0,V}(\la_1,\dots,\la_{p-1}),
\end{align}
where $spline_{\la_1,\dots,\la_{p-1}}$ is a piecewise polynomial of degree $p-2$ with breakpoints $\la_1,\dots,\la_{p-1}$ and $dm_{p-1,H_0,V}(\la_1,\dots,\la_{p-1})$ is a measure on $\Reals^{p-1}$ determined by $p-1$ copies of the spectral measure of $H_0$ intertwined with $p-1$ copies of the perturbation $V$ (see section \ref{sec:mr} for the precise formula). As it is noticed in section \ref{sec:mr}, the function $\eta_2$ given by \eqref{f-la:eta_intro} coincides with the function $\eta_{H_0,H_0+V}$ given by \eqref{f-la:koviaxi}, provided $\tau(|V|)<\infty$. The techniques of \cite{Kop84} that prove existence of $\nu_p$ when $\Mcal=\BH$ do not give absolute continuity of $\nu_p$. We obtain $\eta_p$ by analyzing the Cauchy transform of the measure $\nu_p$ satisfying the trace formula \eqref{f-la:tr} (see section \ref{sec:ct}).

The approach of this paper, developed mainly for higher order
spectral shift functions, contributes to the subject of Krein's spectral shift function as
well. In 1972, using Theorem \ref{prop:Rpf} \eqref{f-la:Rpf} and the double operator integral representation for the derivative
\[\frac{d}{dx}f(H_0+xV)={\int}_{\Reals^2}
\De_{\la_1,\la_2}^{(1)}(f)\,E_{H_0+xV}(d\la_1)VE_{H_0+xV}(d\la_2),\]
M.Sh. Birman and M.Z. Solomyak \cite{Birman72} showed that
\[\tau\big[f(H_0+V)-f(H_0)\big]=\int_\Reals f'(t)
\int_0^1\tau[E_{H_0+xV}(dt)]\,dx\] (see \cite[Theorem 6.3]{Azamov0}
for the analogous result in the context of von Neumann algebras), which along with Krein's trace
formula
\[\tau\big[f(H_0+V)-f(H_0)\big]=\int_\Reals f'(t)\xi_{H_0+V,H_0}(t)\,dt\] \cite{Krein1,Carey,Azamov}
implied the spectral averaging formula
\begin{align}
\label{f-la:sa}
\int_0^1\tau[E_{H_0+xV}(dt)]\,dx=\xi_{H_0+V,H_0}(t)\,dt\end{align}
(see \cite{MakarovL,ms,Simonbs,convexity} for generalizations and extensions).
The operator $f(H_0+V)-f(H_0)$ also admits a double operator integral representation
\begin{align}\label{f-la:doi}
f(H_0+V)-f(H_0)={\int}_{\Reals^2}
\De_{\la_1,\la_2}^{(1)}(f)\,E_{H_0+V}(d\la_1)VE_{H_0}(d\la_2).
\end{align}
A natural question raised by M.Sh. Birman (see, e.g., \cite{BirmanLN}) asks if it is possible to deduce existence of $\xi_{H_0+V,H_0}$, or equivalently, absolute continuity of the measure $\int_0^1\tau[E_{H_0+xV}(dt)]\,dx$, directly from the double operator integral representation
\eqref{f-la:doi}. For $\Mcal$ a finite von Neumann algebra, we answer this question affirmatively and represent $\xi_{H_0+V,H_0}$ as an
integral of a basic spline straightforwardly from \eqref{f-la:doi} (see section \ref{sec:xi}). A general property of a basic spline is that it has the minimal support among all the splines with the same degree, smoothness, and domain properties (see, e.g., \cite{Vore}). When $\dim(\Hcal)<\infty$, higher order spectral shift functions can be written as integrals of basic splines as well (see section \ref{sec:xi}).

By combining different representations for the remainder $\tau[R_{p,H_0,V}(f)]$ in the setting of $\Mcal=\BH$, we prove absolute continuity of the measure \begin{align*}
A\mapsto\int_0^1(1-x)^{p-1}\tau\big[(E_{H_0+xV}(A)V)^p\big]\,dx
\end{align*} and derive higher order analogs of the spectral averaging formula \eqref{f-la:sa} (see section \ref{sec:sa}).

Basic technical tools of the paper are discussed in sections \ref{sec:dd} -- \ref{sec:mm}, main results are stated in section \ref{sec:mr} and then proved in sections \ref{sec:ct} -- \ref{sec:vn}, additional representations for spectral shift functions are obtained in section \ref{sec:xi}, and the Birman-Solomyak spectral averaging formula is generalized in section \ref{sec:sa}.
By saying ``the standard setting" or ``$\tau$ is the standard trace", we implicitly assume that $\Mcal=\BH$ and $\tau$ is the usual trace defined on the trace class operators of $\BH$.
Let $\ncl{p}$ denote the non-commutative $L_p$-space of $(\Mcal,\tau)$ with the norm $\norm{V}_p=\tau(|V|^p)^{1/p}$ and $\ncs{p}=\ncl{p}\cap\Mcal$ the Schatten $p$-class of
$(\Mcal,\tau)$. The Schatten $p$-class is equipped with the norm $\norm{\cdot}_{p,\infty}=\norm{\cdot}_p+\norm{\cdot}$, where $\norm{\cdot}$ is the operator norm. Throughout the paper, $H_0$ and $V$ denote self-adjoint operators in $\Mcal$ or affiliated with $\Mcal$; $V$ is mainly taken to be an element of $\ncs{p}$. Let $\Rfr$ denote the set of rational functions on $\Reals$ with nonreal poles, $\Rfr_b$ the subset of $\Rfr$ of bounded functions. The symbol $f_z$ is reserved for the function $\Reals\ni\la\mapsto\frac{1}{z-\la}$, where $z\in\Complex\setminus\Reals$.

\section{Divided differences and splines}

\label{sec:dd}

\begin{defi}\label{prop:dddef}
The divided difference of order $p$ is an operation on functions $f$ of one (real) variable, which we will usually call $\la$, defined recursively as follows:
\begin{align*}
&\Delta^{(0)}_{\la_1}(f):=f(\la_1),\\
&\dd{\la_{p+1}}{p}{f}:=
\begin{cases}
\frac{\Delta^{(p-1)}_{\la_1,\dots,\la_{p-1},\la_p}(f)-
\Delta^{(p-1)}_{\la_1,\dots,\la_{p-1},\la_{p+1}}(f)}{\la_p-\la_{p+1}}&
\text{ if } \la_p\neq\la_{p+1}\\[2ex]
\frac{\partial}{\partial t}\big|_{t=\la_p}\Delta^{(p-1)}_{\la_1,\dots,\la_{p-1},t}(f)&
\text{ if } \la_p=\la_{p+1}.
\end{cases}
\end{align*}
\end{defi}

Below we state selected facts on the divided difference (see, e.g., \cite{Vore}).

\begin{prop}
\label{prop:dds}
\begin{enumerate}
\item\label{f-la:sim} (See \cite[Section 4.7, (a)]{Vore}.) $\dd{\la_{p+1}}{p}{f}$ is symmetric in $\la_1,\la_2,\dots,\la_{p+1}$.

\item\label{f-la:dd7}(See \cite[Section 4.7, (h)]{Vore}.) If all $\la_1,\la_2,\dots,\la_{p+1}$ are distinct, then
\[\dd{\la_{p+1}}{p}{f}=\sum_{j=1}^{p+1}\frac{f(\la_j)}{\prod_{k\neq j}(\la_j-\la_k)}.\]

\item\label{f-la:dd3} (See \cite[Section 4.7]{Vore}.)
For $f$ a sufficiently smooth function,
\[\dd{\la_{p+1}}{p}{f}=\sum_{i\in
\mathcal{I}}\sum_{j=0}^{m(\la_i)-1}c_{ij}(\la_1,\dots,\la_{p+1})
f^{(j)}(\la_i).\] Here $\mathcal{I}$ is
the set of indices $i$ for which $\la_i$ are distinct, $m(\la_i)$
is the multiplicity of $\la_i$, and $c_{ij}(\la_1,\dots,\la_{p+1})\in\Complex$.

\item\label{f-la:dd2} (See \cite[Section 4.7]{Vore}.)\\
$\dd{\la_{p+1}}{p}{a_p\la^p+a_{p-1}\la^{p-1}+\dots +a_1\la+a_0}=a_p$, where $a_0,a_1,\dots,a_p\in\Complex$.

\item\label{f-la:dd5} (See \cite[Section 5.2, (2.3) and (2.6)]{Vore}.)

The basic spline with the break points
$\la_1,\dots,\la_{p+1}$, where at least two of the values are distinct,
is defined by
\[t\mapsto\begin{cases}
\frac{1}{|\la_2-\la_1|}\chi_{(\min\{\la_1,\la_2\},\max\{\la_1,\la_2\})}(t)&
\text{ if } p=1\\[2ex]
\dd{\la_{p+1}}{p}{(\la-t)^{p-1}_+}& \text{ if } p>1\end{cases}.\]
Here the truncated power is defined by
$x_+^k=\begin{cases}x^k &\text{ if }x\geq 0\\0&\text{ if }x<0,\end{cases}$ for $k\in\Nats$.

The basic spline is non-negative, supported in
\[[\min\{\la_1,\dots,\la_{p+1}\},\max\{\la_1,\dots,\la_{p+1}\}]\] and
integrable with the integral equal to $1/p$. (Often the basic spline is normalized so that its integral equals 1).

\item\label{f-la:dd4} (See \cite[Section 5.2, (2.2) and Section 4.7, (c)]{Vore}.)

For $f\in C^p[\min\{\la_1,\dots,\la_{p+1}\},\max\{\la_1,\dots,\la_{p+1}\}]$,
\begin{align*}&\dd{\la_{p+1}}{p}{f}\\&\quad=
\begin{cases}
\frac{1}{(p-1)!}\int_{-\infty}^\infty
f^{(p)}(t)\dd{\la_{p+1}}{p}{(\la-t)^{p-1}_+}\,dt&\;\text{ if
}\;\exists i_1,i_2\text{ such that } \la_{i_1}\neq \la_{i_2}\\[2ex]
\frac{1}{p!}f^{(p)}(\la_1)&\;\text{ if
}\;\la_1=\la_2=\cdots=\la_{p+1}.
\end{cases}\end{align*}

\item\label{f-la:dd6} (See \cite[Section 4.7, (l)]{Vore}.) Let $f\in C^p[a,b]$. Then, for $\{\la_1,\dots,\la_{p+1}\}\subset[a,b]$,
\[\big|\dd{\la_{p+1}}{p}{f}\big|\leq\frac{1}{p!}\max_{\la\in [a,b]}|f^{(p)}(\la)|.\]
\end{enumerate}
\end{prop}

Below we state useful properties of the divided difference to be
used in the paper.

\begin{lemma} \label{f-la:dd1} For $z\in\Complex$, with $\im(z)\neq 0$,
\[\dd{\la_{p+1}}{p}{\frac{1}{z-\la}}=\prod_{j=1}^{p+1}\frac{1}{z-\la_j},\]
where the divided difference is taken with respect to the real variable $\la$.
\end{lemma}

\begin{proof}
We notice that by Definition \ref{prop:dddef},
\begin{align*}
&\De_{\la_1,\la_2}^{(1)}\left(\frac{1}{z-\la}\right)=
\begin{cases}
\left(\frac{1}{z-\la_1}-\frac{1}{z-\la_2}\right)
\frac{1}{\la_1-\la_2}=
\frac{1}{(z-\la_1)(z-\la_2)}&\text{ if }\la_1\neq \la_2\\[2ex]
\frac{\partial}{\partial t}\big|_{t=\la_1}\left(\frac{1}{z-t}\right)=
\frac{1}{(z-\la_1)^2}=\frac{1}{(z-\la_1)(z-\la_2)}&\text{ if }\la_1=\la_2.
\end{cases}
\end{align*}By repeating the same argument, we obtain
\[\De_{\la_1,\la_2,\la_3}^{(2)}
\left(\frac{1}{z-\la}\right)=\frac{1}{(z-\la_1)(z-\la_2)(z-\la_3)}.\]
The rest of the proof is accomplished by induction.
\end{proof}

\begin{lemma}\label{prop:dd_der}
Let $D$ be a domain in $\Complex$ and $f$ a function continuously differentiable sufficiently many times on $D\times\Reals$.
Then for $p\in\Nats$,\\
(i)\[\int\dd{\la_{p+1}}{p}{f(z,\la)}\,dz=\dd{\la_{p+1}}{p}{\int f(z,\la)\,dz},\] with an appropriate choice of the constant of integration on the left-hand side;\\
(ii)
\[\lim_{z\rightarrow z_0}\dd{\la_{p+1}}{p}{f(z,\la)}=\dd{\la_{p+1}}{p}{\lim_{z\rightarrow z_0} f(z,\la)},\quad z_0\in D;\]
(iii)\[\frac{\partial}{\partial
z}\left[\dd{\la_{p+1}}{p}{f(z,\la)}\right]=\dd{\la_{p+1}}{p}{\frac{\partial}{\partial
z}f(z,\la)},\] where the divided difference is taken with respect to the variable $\la$.
\end{lemma}

\begin{proof}
Follows immediately from Proposition \ref{prop:dds} \eqref{f-la:dd3}.
\end{proof}

\begin{cor}\label{prop:dd_rat} For $p,k\in\Nats$,
\[\frac{(-1)^k}{k!}\frac{\partial^k}{\partial z^k}\left(\prod_{j=1}^{p+1}\frac{1}{z-\la_j}\right)
=\dd{\la_{p+1}}{p}{\frac{1}{(z-\la)^{k+1}}}.\]
\end{cor}

\begin{proof}
Follows immediately from Lemma \ref{f-la:dd1} and Lemma \ref{prop:dd_der}.
\end{proof}

\section{Remainders of Taylor-type approximations}
\label{sec:r}

In this section, we collect technical facts on derivatives of operator functions and remainders of the Taylor-type approximations.

The following lemma is routine.

\begin{lemma}\label{prop:res'}Let $H_0=H_0^*$ be an operator in $\Hcal$ and $V=V^*\in\BH$. Let $H_x=H_0+xV$, with $x\in\Reals$. Then,
\[\frac{d^p}{dx^p}\big((zI-H_x)^{-k}\big)=p!\sum_{\substack{1\leq k_0,k_1,\dots,k_p\leq k\\k_0+k_1+\dots +k_p=k+p}}(zI-H_x)^{-k_0}V(zI-H_x)^{-k_1}V\dots V(zI-H_x)^{-k_p}.\]
If, in addition, $H_0$ is bounded, then
\[\frac{d^p}{dx^p}\big(H_x^{k}\big)=p!\sum_{\substack{0\leq k_0,k_1,\dots,k_p\\k_0+k_1+\dots +k_p=k-p}}H_x^{k_0}VH_x^{k_1}V\dots VH_x^{k_p},\quad p\leq k.\]
\end{lemma}

\begin{lemma}\label{prop:d=} Let $H_0=H_0^*$ be an operator affiliated with $\Mcal$ and $V=V^*\in\ncs{2}$. Then,
\begin{align}\label{f-la:d=}
\frac{(-1)^k}{k!}\frac{d^k}{dz^k}\tau\big[(zI-H_0)^{-1}V(zI-H_0)^{-1}V(zI-H_0)^{-1}\big]
\\\nonumber=\frac12\tau\left[\frac{d^2}{dx^2}\bigg|_{x=0}\big((zI-H_0-xV)^{-k-1}\big)\right]
\end{align}
\end{lemma}

\begin{proof}Firstly, we compute the left-hand side of \eqref{f-la:d=}.
By cyclicity of the trace,
\[\tau\big[(zI-H_0)^{-1}V(zI-H_0)^{-1}V(zI-H_0)^{-1}\big]=
\tau\big[(zI-H_0)^{-2}V(zI-H_0)^{-1}V\big].\]
By continuity of the trace in the norm $\norm{\cdot}_{1,\infty}$,
\begin{align*}
&\frac{d}{dz}\tau\big[(zI-H_0)^{-2}V(zI-H_0)^{-1}V\big]\\
&\quad=\tau\left[\frac{d}{dz}\left((zI-H_0)^{-2}V(zI-H_0)^{-1}V\right)\right].
\end{align*}
It is easy to see that
\begin{align}\label{f-la:c1}\nonumber
&\frac{d^k}{dz^k}\left((zI-H_0)^{-2}V(zI-H_0)^{-1}V\right)\\\nonumber
&=\sum_{j=0}^k\frac{k!}{j!(k-j)!}(-1)^j(j+1)!(zI-H_0)^{-2-j}V
(-1)^{k-j}(k-j)!(zI-H_0)^{-1-(k-j)}V\\
&=(-1)^k k!\sum_{j=0}^k(j+1)(zI-H_0)^{-2-j}V(zI-H_0)^{-1-(k-j)}V.
\end{align}
Now we compute the right-hand side of \eqref{f-la:d=}. Let $H_x=H_0+xV$.
It is routine to see that
\begin{align*}
\frac{d}{dx}\big((zI-H_x)^{-(k+1)}\big)=\sum_{i=1}^{k+1}(zI-H_x)^{-i}V(zI-H_x)^{-(k+2-i)},
\end{align*} and hence,
\begin{align*}
&\frac{d^2}{dx^2}\bigg|_{x=0}\big((zI-H_x)^{-(k+1)}\big)\\
&\quad=2\sum_{i=1}^{k+1}\frac{d}{dx}\bigg|_{x=0}\big((zI-H_x)^{-i}\big)V(zI-H_0)^{-(k+2-i)}\\
&\quad=2\sum_{i=1}^{k+1}\sum_{j=0}^{i-1}(zI-H_0)^{-(i-j)}V(zI-H_0)^{-1-j}V(zI-H_0)^{-(k+2-i)}.
\end{align*}Multiplying by $1/2$ and evaluating the trace in the latter expression provides
\begin{align*}
&\frac12\tau\left[\frac{d^2}{dx^2}\bigg|_{x=0}\big((zI-H_0-xV)^{-k-1}\big)\right]\\
&\quad=\sum_{i=1}^{k+1}\sum_{j=0}^{i-1}\tau\big[(zI-H_0)^{-1-j}V(zI-H_0)^{-(k+2-j)}V\big]\\
&\quad=\sum_{j=0}^{k}\sum_{i=j}^{k+1}\tau\big[(zI-H_0)^{-1-j}V(zI-H_0)^{-2-(k-j)}V\big]\\
&\quad=\sum_{j=0}^{k}(k+1-j)\tau\big[(zI-H_0)^{-1-j}V(zI-H_0)^{-2-(k-j)}V\big].
\end{align*}
By changing the index of summation $i=k-j$ in the latter expression and by cyclicity of the trace, we obtain
\begin{align}\label{f-la:c2}
\sum_{i=0}^{k}(i+1)\tau\big[(zI-H_0)^{-2-i}V(zI-H_0)^{-1-(k-i)}V\big]
\end{align}Comparing \eqref{f-la:c1} and \eqref{f-la:c2} completes the proof of the lemma.
\end{proof}

As a particular case of results of \cite{PellerMult} we have the lemma below.

\begin{lemma}\label{prop:dd_as'}Let $H_0=H_0^*$ be an operator in $\Hcal$ and $V=V^*\in\BH$. Denote $H_x=H_0+xV$. For $f\in\Rfr_b$,
\begin{align}\label{f-la:dd_as'}&\frac{d^p}{dx^p}f(H_0+xV)\\\nonumber&\quad=
p!\int_\Reals\int_\Reals\dots\int_\Reals\dd{\la_{p+1}}{p}{f}\,E_{H_x}(d\la_1)VE_{H_x}(d\la_2)V\dots VE_{H_x}(d\la_{p+1}).
\end{align} If, in addition, $H_0$ is bounded, then \eqref{f-la:dd_as'} holds for $f\in\Rfr$.
\end{lemma}

\begin{remark}
It was proved in \cite[Theorem 2.2]{Daletskii} that for $H_0$ a bounded operator, $\frac{d^p}{dt^p}f(H_0+tV)$ is defined when $f\in C^{2p}(\Reals)$ and the derivative can be computed as an iterated operator integral
\eqref{f-la:dd_as'}. It was proved later in \cite{PellerMult} that the G\^ateaux derivative $\frac{d^p}{dt^p}f(H_0+tV)$ is defined for $f$ in the intersection of the Besov classes $B^p_{\infty 1}(\Reals)\cap B^1_{\infty 1}(\Reals)$ and can be computed as a Bochner-type multiple operator integral.
\end{remark}

The following lemma is a straightforward consequence of Lemma \ref{prop:res'}.

\begin{lemma}\label{prop:pol}
Let $H_0=H_0^*$ be an operator in $\Hcal$ and $V=V^*\in\BH$. Then for $f$ a
polynomial of degree $m$,
\[R_{p,H_0,V}(f)=\sum_{\substack{k_0,k_1\dots,k_p\geq 0\\ k_0+k_1+\dots +k_p=m-p}}a_{k_0,k_1,\dots,k_p}
H_0^{k_0}VH_0^{k_1}V\dots VH_0^{k_p},\] with $a_{k_0,k_1,\dots,k_p}$
numbers.
\end{lemma}

\begin{lemma}\label{prop:res}Let $H_0=H_0^*$ be an operator in $\Hcal$ and $V=V^*\in\BH$.
Then,
\begin{align}\label{f-la:g_Rres1}
R_{p,H_0,V}(f_z)&=(zI-H_0-V)^{-1}-
\sum_{j=0}^{p-1}(zI-H_0)^{-1}\left(V(zI-H_0)^{-1}\right)^j\\\label{f-la:g_Rres2}
&=(zI-H_0-V)^{-1}\left(V(zI-H_0)^{-1}\right)^p.
\end{align}
\end{lemma}

\begin{proof}
By Lemma \ref{prop:res'},
\[\frac{d^j}{dx^j}\bigg|_{x=x_0}\left((zI-H_0-xV)^{-1}\right)=
j!(zI-H_0-x_0V)^{-1}\left(V(zI-H_0-x_0V)^{-1}\right)^j,\] which gives
\eqref{f-la:g_Rres1}.
To derive \eqref{f-la:g_Rres2} from \eqref{f-la:g_Rres1}, we use
repeatedly the resolvent identity
\[(zI-H_0-V)^{-1}-(zI-H_0)^{-1}=(zI-H_0-V)^{-1}V(zI-H_0)^{-1}.\]
By combining $(zI-H_0-V)^{-1}$ and the first summand of
\[\sum_{j=0}^{p-1}(zI-H_0)^{-1}\left(V(zI-H_0)^{-1}\right)^j,\] we
obtain that (provided $p>1$)
\begin{align*}&(zI-H_0-V)^{-1}-\sum_{j=0}^{p-1}(zI-H_0)^{-1}\left(V(zI-H_0)^{-1}\right)^j\\
&\quad=(zI-H_0-V)^{-1}V(zI-H_0)^{-1}-\sum_{j=1}^{p-1}(zI-H_0)^{-1}\left(V(zI-H_0)^{-1}\right)^j.
\end{align*}
Repeating the reasoning above sufficiently many times completes the proof of \eqref{f-la:g_Rres2}.
\end{proof}

From \eqref{f-la:g_Rres1} we have the following relation between the remainders of different order.

\begin{lemma}\label{prop:R_rec}
Let $H_0=H_0^*$ be an operator in $\Hcal$ and $V=V^*\in\BH$.
Then
\[R_{p+1,H_0,V}\left(f_z\right)=R_{p,H_0,V}\left(f_z\right)
-\big((zI-H_0)^{-1}V\big)^p(zI-H_0)^{-1}.\]
\end{lemma}

The following lemma is a straightforward generalization of
\cite[Lemma 2.6]{Dostanic}.

\begin{lemma}\label{prop:fcalc}
Let $H_0=H_0^*, V=V^*\in\BH$, and $\Gamma=\{\la\st
|\la|=1+\norm{H_0}+\norm{V}\}$. Then, for every function $f$ analytic
in a neighborhood of $D=\{\la\st |\la|\leq 1+\norm{H_0}+\norm{V}\}$,
\begin{align*}R_{p,H_0,V}(f)=
\frac{1}{2\pi\i}\oint_\Gamma
f(\la)(\la I-H_0)^{-1}\big(V(\la I-H_0)^{-1}\big)^p\big(I-V(\la I-H_0)^{-1}\big)^{-1}\,d\la.
\end{align*}
\end{lemma}



Let $(S,\nu)$ be a measure space and let $\mathcal{L}_\infty^{so^*}(S,\nu,\ncs{1})$ denote the $*$-algebra of $\norm{\cdot}$-bounded $so^*$-measurable functions $F:S\mapsto\ncs{1}$ \cite{Pagter}.

\begin{prop}$($See \cite[Lemma 3.10]{Azamov0}$).$\label{prop:so*}
Let $F$ be a function in $\mathcal{L}_\infty^{so^*}(S,\nu,\ncs{1})$ uniformly
$\ncs{1}$-bounded. Then $\int_S F(s)\,d\nu(s)\in\ncs{1}$,
$\tau(F(\cdot))$ is measurable and
\[\tau\left(\int_S F(s)\,d\nu(s)\right)=\int_S \tau(F(s))\,d\nu(s).\]
\end{prop}

Similarly to \cite[Lemma 4.5, Theorem 5.7]{Azamov0}, we have the following differentiation formula for an operator function $f(\cdot)$, with $f\in\mathcal{W}_p$.

\begin{lemma}
\label{prop:eder} Let $H_0=H_0^*$ be an operator affiliated with $\Mcal$ and $V=V^*\in\ncs{p}$. Let $H_x=H_0+xV$, with $x\in\Reals$. Then, for $f\in\mathcal{W}_p$ given by $f(\la)=\int_\Reals e^{\i t\la}\,d\mu_f(t)$, the function $f(H_x)$ is $p$ times Fr\'{e}chet differentiable in the norm $\norm{\cdot}_{1,\infty}$ and the derivative equals the Bochner-type multiple operator integral
\begin{align*}
\frac{d^p}{dx^p}f(H_x)=p!\int_{\Pi^{(p)}}e^{\i
(s_0-s_1)H_x}V\dots V e^{\i (s_{p-1}-s_p)H_x}Ve^{\i s_p H_x}\,d\sigma_f^{(p)}(s_0,\dots,s_p).
\end{align*} Here
\[{\Pi^{(p)}}=\{(s_0,s_1,\dots,s_p)\in\Reals^{p+1}\st |s_p|\leq\dots\leq|s_1|\leq|s_0|,
\text{\rm sign}(s_0)=\dots=\text{\rm sign}(s_p)\}\] and
$d\sigma_f^{(p)}(s_0,s_1,\dots,s_p)=\i^p\mu_f(ds_0)ds_1\dots ds_p$.

In particular, for $t\in\Reals$,
\begin{align*}
\frac{d^p}{dx^p}e^{\i tH_x}=\i^p p!\int_{\Pi^{(p-1)}}e^{\i
(t-s_1)H_x}V\dots V e^{\i (s_{p-1}-s_p)H_x}Ve^{\i s_p H_x}\,ds_p\dots ds_1.
\end{align*}
\end{lemma}

By applying Proposition \ref{prop:so*} and Lemma \ref{prop:eder}, we obtain the following

\begin{lemma}\label{prop:edertr}
Let $H_0=H_0^*$ be an operator affiliated with $\Mcal$ and $V=V^*\in\ncs{p}$. Let $H_x=H_0+xV$, with $x\in\Reals$. Then for $f\in\mathcal{W}_p$, we have $\frac{d^p}{dx^p}f(H_x)\in\ncs{1}$,
\begin{align*}
\tau\left[\frac{d^p}{dx^p}f(H_x)\right]=p!\int_{\Pi^{(p)}}\tau\left[e^{\i
(s_0-s_1)H_x}V\dots V e^{\i (s_{p-1}-s_p)H_x}Ve^{\i s_p H_x}\right]\,d\sigma_f^{(p)}(s_0,\dots,s_p)
\end{align*} and
\[\norm{\frac{d^p}{dx^p}f(H_x)}_1\leq\norm{V}_p^p\norm{\mu_{f^{(p)}}}_p.\]
\end{lemma}

\begin{cor}\label{prop:dos_der}
Under the assumptions of Lemma \ref{prop:edertr},
\[\tau\left[\frac{d^p}{dx^p}f(H_x)\right]=\int_\Reals\tau\left[\frac{d^p}{dx^p}e^{\i tH_x}\right]\,d\mu_f(t).\]
\end{cor}

\begin{proof}
The claim is proved by reducing the double integral to an iterated one and applying Lemmas \ref{prop:eder} and \ref{prop:edertr}.
\end{proof}

\begin{remark}
By combining Lemma \ref{prop:edertr} and Theorem \ref{prop:Rpf} \eqref{f-la:Rpf}, one obtains
the estimate \eqref{f-la:mest}.
\end{remark}

\section{Multiple spectral measures}

\label{sec:mm}

We will need the fact that certain finitely additive ``multiple
spectral measures" extend to countably additive measures.

\begin{thm}\label{prop:m}
Let $2\leq p\in\Nats$ and let $E_1,E_2,\ldots,E_p$ be projection-valued
Borel measures from $\Reals$ into $\Mcal$. Suppose that
$V_1,\ldots,V_p$ belong to $\ncs{2}$. Assume that either $\tau$ is the
standard trace or $p=2$. Then there is a unique (complex) Borel measure
$m$ on $\Reals^p$ with total variation not exceeding the product
$\norm{V_1}_2\norm{V_2}_2\cdots\,\norm{V_p}_2$, whose
value on rectangles is given by
\[m(A_1\times A_2\times\cdots\times A_p)
=\tau\big[E_1(A_1)V_1E_2(A_2)V_2\dots V_{p-1}E_p(A_p)V_p\big]\]
for all Borel subsets $A_1,A_2,\ldots,A_p$ of $\Reals$.
\end{thm}

\begin{proof}
It is enough (see, e.g., \cite[Theorem 2.12]{bimeasure} for $p=2$) to prove that the variation of the set function $m$ on the rectangles of $\Reals^p$ is bounded by $\norm{V_1}_2\norm{V_2}_2\cdots\,\norm{V_p}_2$, which can be accomplished completely analogously to the proof of \cite[Theorem 1]{Pavlov} (see also \cite{Birman96}).
\end{proof}


\begin{remark}
For $\tau$ the standard trace, the bound for the total variation in Theorem \ref{prop:m} was proved in \cite[Theorem 1]{Pavlov}. Theorem \ref{prop:m} with $\tau$ standard was also obtained in \cite{Birman96}. The proof in \cite{Birman96} is based on the facts that a Hilbert-Schmidt operator can be approximated by finite-rank operators in the norm $\norm{\cdot}_2$ and that for rank-one perturbations $V_1,\dots,V_p$ and $\tau$ the standard trace, the set function $m$ decomposes into a product of scalar measures. It is classical that a direct product of countably additive measures always has a countably-additive extension to the $\sigma$-algebra generated by the direct product of the $\sigma$-algebras involved. The argument of \cite{Birman96} cannot be directly extended to the case of a general trace. For a general trace $\tau$, the set function $m$ is known to be of bounded variation only if $p=2$. Technically, this constraint is explained by the fact that in general $\norm{\cdot}_p$ is not dominated by $\norm{\cdot}_2$, as distinct from the particular case of  the standard trace $\tau$. A counterexample constructed further in this section demonstrates that $p=2$ is not only a technical constraint.
\end{remark}

\begin{cor} \label{prop:m.x}
Let $2\leq p\in\Nats$ and let $E_1,E_2,\ldots,E_p$ be projection-valued
Borel measures from $\Reals$ to $\Mcal$. Suppose that
$V_1,\ldots,V_p$ belong to $\ncs{2}$. Assume that either $\tau$ is the
standard trace or $p=2$. Then there is a unique (complex) Borel measure
$m_1$ on $\Reals^{p+1}$ with total variation not exceeding the product
$\norm{V_1}_2\norm{V_2}_2\cdots\,\norm{V_p}_2$, whose
value on rectangles of $\Reals^{p+1}$ is given by
\[m_1(A_1\times A_2\times\cdots\times A_p\times A_{p+1})
=\tau\big[E_1(A_1)V_1E_2(A_2)V_2\dots V_{p-1}E_p(A_p)V_p E_1(A_{p+1})\big],\]
for all Borel subsets $A_1,A_2,\ldots,A_p,A_{p+1}$ of $\Reals$.
\end{cor}

\begin{proof}
It is straightforward to see that
\[m_1(A_1\times A_2\times\cdots\times A_p\times A_{p+1})
=\tau\big[E_1(A_1\cap A_{p+1})V_1E_2(A_2)V_2\dots V_{p-1}E_p(A_p)V_p\big].\]
By repeating the argument of \cite{Birman96,Pavlov}, one can see that the total variation of the set function $m_1$ is bounded on the rectangles of $\Reals^{p+1}$ by $\norm{V_1}_2\norm{V_2}_2\cdots\,\norm{V_p}_2$. Thus, $m_1$ extends to a unique complex Borel measure on $\Reals^{p+1}$ with variation bounded by $\norm{V_1}_2\norm{V_2}_2\cdots\,\norm{V_p}_2$.
\end{proof}

\begin{cor}\label{prop:m.k}
Let $2\leq p\in\Nats$, $E_1,\ldots,E_p$ projection-valued
Borel measures from $\Reals$ into $\Mcal$, and $\tau$ a finite trace.
Suppose that $V_1,\ldots,V_{p-1}$ belong to $\Mcal$. Assume that
either $\tau$ is the standard trace or $p=2$. Then there is a unique
complex Borel measure $m_2$ on $\Reals^p$ with total variation not
exceeding
$\norm{V_1}_2\norm{V_2}_2\cdots\,\norm{V_{p-1}}_2\tau(I)^{1/2}$,
whose value on rectangles is given by
\[m_2(A_1\times A_2\times\cdots\times
A_p)=\tau\big[E_1(A_1)V_1E_2(A_2)V_2\dots V_{p-1}E_p(A_p)\big]\]
for all Borel subsets $A_1,A_2,\ldots,A_p$ of $\Reals$.
\end{cor}
\begin{proof}
It is an immediate consequence of Theorem \ref{prop:m} applied to
$V_p=I$.
\end{proof}

In the sequel, we will work with the set functions
\begin{align*}
m_{p,H_0,V}(A_1\times A_2\times\cdots\times A_p)
=\tau\big[E_{H_0}(A_1)VE_{H_0}(A_2)V\dots VE_{H_0}(A_p)V\big],
\end{align*}
\begin{align*}
&m_{p,H_0,V}^{(1)}(A_1\times A_2\times\cdots\times A_p\times A_{p+1})\\
&\quad=\tau\big[E_{H_0}(A_1)VE_{H_0}(A_2)V\dots VE_{H_0}(A_p)VE_{H_0}(A_{p+1})\big],
\end{align*}
\begin{align*}
&m_{p,H_0,V}^{(2)}(A_1\times A_2\times\cdots\times A_p\times A_{p+1})\\
&\quad=\tau\big[E_{H_0+V}(A_1)VE_{H_0}(A_2)V\dots VE_{H_0}(A_p)VE_{H_0}(A_{p+1})\big],
\end{align*} and their countably-additive extensions (when they exist). Here $A_j$ are measurable subsets of $\Reals$, $H_0=H_0^*$ is affiliated with $\Mcal$, and $V=V^*\in\ncs{2}$.

In the next result, freeness of $(zI-H_0)^{-1}$ and $V$ means freeness of the algebra generated by the spectral projections of $H_0$ and the unital algebra generated by $V$.
\begin{thm}\label{prop:m_free}
Let $\tau$ be a finite trace normalized by $\tau(I)=1$ and let $H_0=H_0^*$ be affiliated with $\Mcal$ and $V=V^*\in\Mcal$. Assume that $(zI-H_0)^{-1}$ and $V$ are free. Then the set functions $m_{p,H_0,V}$ and $m_{p,H_0,V}^{(1)}$ extend to countably additive measures of bounded variation.
\end{thm}

\begin{proof} We prove the claim for the function $m_{p,H_0,V}$; the case of $m_{p,H_0,V}^{(1)}$ is completely analogous.
Using the moment--cumulant formula (see \cite[Theorem~2.17]{Speicher}), and that $\prod_iE_{H_0}(A_i)=E_{H_0}(\cap_iA_i)$
we have
\begin{align}\label{f-la:mfree*}
m_{p,H_0,V}(A_1\times\dots\times A_p)&=\tau\big[E_{H_0}(A_1)V\dots E_{H_0}(A_p)V\big]\\\nonumber
&=\sum_{\pi=\{B_1,\ldots,B_\ell\}\in\operatorname{NC}(p)} k_{K(\pi)}[V,\ldots,V]\prod_{j=1}^\ell\tau\big(E_{H_0}(\cap_{i\in B_j}A_i)\big),
\end{align}
where $\operatorname{NC}(p)$ is the lattice of all noncrossing partitions of $\{1,\ldots,p\}$ and where
$k_{K(\pi)}[V,\ldots,V]$ is the product of cumulants of V, associated to the block structure of the Kreweras
complement $K(\pi)$ of $\pi$;
thus, $k_{K(\pi)}[V,\ldots,V]$ is equal to a polynomial (that depends on $\pi$) in $p$ variables,
evaluated at $\tau(V),\tau(V^2),\ldots,\tau(V^p)$.
Given $\pi=\{B_1,\ldots,B_\ell\}\in\operatorname{NC}(p)$, the measure
\begin{align}
\label{f-la:mfree*1}
\gamma_{p,\pi}:A_1\times\cdots\times A_p\mapsto\prod_{j=1}^\ell\tau\big(E_{H_0}(\cap_{i\in B_j}A_i)\big)
\end{align}
is the push--forward of the $\ell$--fold product $\times_1^\ell(\tau\circ E_{H_0})$
of the spectral distribution measure of $H_0$ (with respect to $\tau$)
under the mapping of $\Reals^\ell$ onto the product of diagonals in $\Reals^p$ according to the block structure $B_1,\ldots,B_\ell$.
Each such push--forward is a probability measure.
Thus, we see that $m_{p,H_0,V}$ is a linear combination of probability measures, and has finite total variation.
\end{proof}

In some cases used in the paper, the measure $m_2$ is known to be real-valued and the measure $m$ non-negative.

\begin{lemma}\label{prop:R_spline}Let $\tau$ be a finite trace. Let $H_0=H_0^*$ be affiliated with $\Mcal$ and $V=V^*\in\Mcal$. Then the measure $m_{1,H_0,V}^{(2)}$ is real-valued.
\end{lemma}

\begin{proof}
For arbitrary measurable subsets $A_1$ and $A_2$ of $\Reals$,
\begin{align*}&\tau\big[E_{H_0+V}(A_1)VE_{H_0}(A_2)\big]\\&\quad=
\tau\big[E_{H_0+V}(A_1)(H_0+V)E_{H_0}(A_2)\big]-
\tau\big[E_{H_0+V}(A_1)H_0 E_{H_0}(A_2)\big]\\&\quad=
\tau\big[E_{H_0}(A_2)\big(E_{H_0+V}(A_1)(H_0+V)\big)E_{H_0}(A_2)\big]\\&\quad\quad-
\tau\big[E_{H_0+V}(A_1)\big(H_0 E_{H_0}(A_2)\big)E_{H_0+V}(A_1)\big],
\end{align*} where the operators
\[E_{H_0}(A_2)\big(E_{H_0+V}(A_1)(H_0+V)\big)E_{H_0}(A_2)\;\text{ and
}\; E_{H_0+V}(A_1)\big(H_0 E_{H_0}(A_2)\big)E_{H_0+V}(A_1)\] are
self-adjoint. Therefore, $m_{1,H_0,V}^{(2)}(A_1\times A_2)\in\Reals$, and the extension of
$m_{1,H_0,V}^{(2)}$ to the Borel subsets of $\Reals^2$ is real-valued.
\end{proof}

\begin{lemma}\label{prop:m_positive}
Let $H_0=H_0^*$ be affiliated with $\Mcal$ and $V=V^*\in\ncs{2}$. Then the measure $m_{2,H_0,V}$ is non-negative.
\end{lemma}

\begin{proof}
For arbitrary measurable subsets $A_1$ and $A_2$ of $\Reals$,
\[\tau\big[E_{H_0}(A_1)VE_{H_0}(A_2)V\big]=
\tau\big[E_{H_0}(A_1)VE_{H_0}(A_2)VE_{H_0}(A_1)\big]\geq 0,\] since
\[\left<E_{H_0}(A_1)VE_{H_0}(A_2)VE_{H_0}(A_1)f,f\right>=
\left<E_{H_0}(A_2)\big(VE_{H_0}(A_1)f\big),\big(VE_{H_0}(A_1)f\big)\right>\geq
0,\] for any $f\in\Hcal$.
\end{proof}

\begin{lemma}\label{prop:no_atoms}Let $\tau$ be a finite trace. Let $H_0=H_0^*$ be an operator affiliated with $\Mcal$ and $V=V^*\in\Mcal$.
Then $m_{k,H_0,V}^{(2)}$ has no atoms on the diagonal $D_{k+1}=\{(\la_1,\la_2,\dots,\la_{k+1})\st\la_1=\la_2=\cdots=\la_{k+1}\in\Reals\}$ of $\Reals^{k+1}$.
\end{lemma}

\begin{proof}By definition of the measure $m_{k,H_0,V}^{(2)}$,
\[m_{k,H_0,V}^{(2)}(\{(\la,\la,\dots,\la)\})=
\tau\left[E_{H_0+V}(\{\la\})VE_{H_0}(\{\la\})\big(VE_{H_0}(\{\la\})\big)^{k-1}\right].\]
We will show that $E_{H_0+V}(\{\la\})VE_{H_0}(\{\la\})$ is the zero
operator.

Let $g$ be an arbitrary vector in $\Hcal$ and let $h=E_{H_0}(\{\la\})g$. Then
$H_0 h=\la h$ and
\begin{align*}&E_{H_0+V}(\{\la\})VE_{H_0}(\{\la\})g\\&=E_{H_0+V}(\{\la\})Vh=
E_{H_0+V}(\{\la\})(H_0+V)h-E_{H_0+V}(\{\la\})H_0 h\\
&=E_{H_0+V}(\{\la\})(H_0+V)h-\la E_{H_0+V}(\{\la\})h
=(H_0+V-\la I)E_{H_0+V}(\{\la\})h=0.\end{align*}
\end{proof}

Upon evaluating a trace, some iterated operator integrals can be written as Lebesgue
integrals with respect to a ``multiple spectral measure".

\begin{lemma}\label{prop:i_to_m_g}Assume the hypothesis of Theorem \ref{prop:m}. Assume that the spectral measures $E_1,E_2,\dots,E_p$ correspond to self-adjoint operators $H_0,H_1,\dots,H_p$ affiliated with $\Mcal$, respectively, and that $V_1,V_2,\dots,V_p\in\ncs{2}$. Let $f_1,f_2,\dots,f_p$ be functions in $C_0^\infty(\Reals)$ (vanishing at infinity). Then
\[\tau[f_1(H_1)V_1 f_2(H_2)V_2\dots f_p(H_p)V_p]=
\int_{\Reals^p}f_1(\la_1)f_2(\la_2)\dots f_p(\la_p)\,dm(\la_1,\la_2,\dots,\la_p),\] with $m$ as in Theorem \ref{prop:m}.
\end{lemma}

\begin{proof}
The result obviously holds for $f_1,f_2,\dots,f_p$ simple functions. Uniform approximation of $f_1,f_2,\dots,f_p\in C_0^\infty(\Reals)$ by (totally bounded) simple functions completes the proof.
\end{proof}

\begin{remark}\label{prop:imgr}
(i) The result analogous to the one of Theorem \ref{prop:m} holds for integrals with respect to the measures $m_1$ and $m_2$.\\
(ii) When the operators $H_0,H_1,\dots,H_p$ are bounded, the functions $f_1,f_2,\dots,f_p$ can be taken in $C^\infty(\Reals)$.
\end{remark}

\begin{cor}\label{prop:i_to_m} Let $H_0=H_0^*$ be affiliated with $\Mcal$ and $V=V^*\in\ncs{2}$. Denote $H_x:=H_0+xV$, $x\in [0,1]$. Assume that either $\tau$ is standard or $p=2$. Then for $f\in\Rfr_b$,
\begin{align*}\tau\left[\frac{d^p}{dx^p}f(H_0+xV)\right]=
p!\int_{\Reals^{p+1}}
\dd{\la_{p+1}}{p}{f}\,dm_{p,H_x,V}^{(1)}(\la_1,\la_2,\dots,\la_{p+1}).\end{align*}
\end{cor}

\begin{proof} It is enough to prove the result for $f(\la)=(z-\la)^{-k}$, $k\in\Nats$.
By Lemma \ref{prop:dd_as'},
\begin{align}\label{f-la:i_to_m.1}&\frac{d^p}{dx^p}\big((zI-H_x)^{-k}\big)\\\nonumber&\quad=
p!\int_\Reals\int_\Reals\dots\int_\Reals
\dd{\la_{p+1}}{p}{(z-\la)^{-k}}\,E_{H_x}(d\la_1)VE_{H_x}(d\la_2)V\dots VE_{H_x}(d\la_{p+1}).
\end{align}
By Corollary \ref{prop:dd_rat}, the function $\dd{\la_{p+1}}{p}{(z-\la)^{-k}}$ is a linear combination of products $\prod_{j=1}^{p+1}f_j(\la_j)$, with $f_j$ in $C_0^\infty(\Reals)$ for $1\leq j\leq p+1$, and hence, the trace of the expression in \eqref{f-la:i_to_m.1} can be written as a linear combination of integrals like in Remark \ref{prop:imgr} (i), where $m_1=m_{p,H_x,V}^{(1)}(\la_1,\la_2,\dots,\la_{p+1})$.
\end{proof}

\begin{remark}\label{prop:imr} One also has
\begin{align*}\tau\left[((zI-H_0)^{-1}V)^p\right]=
\int_{\Reals^{p}}
\dd{\la_{p}}{p-1}{f_z}\,dm_{p,H_0,V}(\la_1,\la_2,\dots,\la_p)
\end{align*} and
\begin{align*}&\tau\left[(zI-H_0-V)^{-1}(V(zI-H_0)^{-1})^p\right]\\
&\quad=\int_{\Reals^{p+1}}
\dd{\la_{p+1}}{p}{f_z}\,dm_{p,H_0,V}^{(2)}(\la_1,\la_2,\dots,\la_{p+1}).
\end{align*}
\end{remark}

\begin{remark}
\label{prop:imrp}
If $H_0$ is bounded, then Corollary \ref{prop:i_to_m} also holds for $f$ a polynomial and for
$f\in C^\infty(\Reals)$ such that $f\big|_{[a,b]}$ is a polynomial, where $[a,b]\supset\sigma(H_0)\cup\sigma(H_0+V)$.
\end{remark}

\bigskip

{\bf A counterexample.} Let $p\ge2$ be an integer.
Let $V$ be a self--adjoint operator on a Hilbert space $\Hcal$
and assume $V$ belongs to the Schatten $p$-class, with respect to the
usual trace $\Tr$.
Let $E$ be a spectral measure.
A crucial estimate in~\cite{Kop84} is of the total variation of the
function that is defined on product sets by
\[
A_1\times\cdots\times A_p\mapsto\Tr\big(E(A_1)VE(A_2)V\cdots E(A_p)V\big).
\]
Unfortunately, the estimate result in~\cite{Kop84} is false when $p\ge3$.
In this section, we provide an example, based on Hadamard matrices, having unbounded total variation.
We also give a version for finite traces.

Consider the self--adjoint unitary $2\times2$ matrix
\[
V_2=\frac1{\sqrt 2}
\begin{pmatrix}
1&1 \\ 1&-1
\end{pmatrix}\,.
\]
When $n=2^k$, consider the self--adjoint unitary $n\times n$ matrix
\[
V_n=\underset{k\mbox{ times}}{\underbrace{V_2\otimes\cdots\otimes V_2}}.
\]
Then each such $V_n$ is a multiple of a Hadamard matrix.

Let $(e_{jk})_{1\le j,k\le n}$ be the usual system of matrix units for $M_n(\Complex)$.
Let $E_n$ be the spectral measure on the set $\{1,\ldots,n\}$ taking
values in $M_n(\Complex)$, defined by $E_n(\{j\})=e_{jj}$.

The following lemma can be proved directly for all $n$, or first in the case $n=2$ and then by observing
how the total variation behaves under taking tensor products of matrices and spectral measures.
\begin{lemma}\label{lem:Unvar}
For every integer $p\ge2$, and every $n$ that is a power of $2$, the set function
\[
A_1\times\cdots\times A_p\mapsto\Tr_n(E_n(A_1)V_nE_n(A_2)V_n\cdots E_n(A_p)V_n)
\]
has total variation $n^{p/2}$.
\end{lemma}

Consider the von Neumann algebra $\Mcal$ with normal trace $\tau$ given by
\[
(\Mcal,\tau)=\bigoplus_{k=1}^\infty\smdb{M_{n(k)}(\Complex)}{\alpha(k)}\,,
\]
where $\alpha(k)>0$ and this notation indicates that $\Mcal$ is the $\ell^\infty$--direct
sum of the matrix algebras $M_{n(k)}(\Complex)$, where every $n(k)$ is a power of $2$ and where $\tau$ is the trace
determined by
\[
\tau(\underset{k-1\text{ times}}{\underbrace{0\oplus\cdots\oplus0}}
 \oplus I_{n(k)}\oplus0\oplus0\oplus\cdots)=\alpha(k).
\]
We will be interested in the following two cases:
\begin{itemize}
\item[(I)] $\Mcal$ embeds in $\BH$, for $\Hcal$ a separable, infinite-dimensional Hilbert
space, in such a way that $\tau$ is the restriction of the usual trace $\Tr$ on $\BH$;
this is equivalent to $\alpha(k)$ being an integer multiple of $n(k)$ for all $k$.
\item[(II)] $\Mcal$ is a finite von Neumann algebra and $\tau$ is normalized to take value $1$ on
the identity;
this is equivalent to $\sum_{k=1}^\infty\alpha(k)=1$.
\end{itemize}
\begin{example}
Consider
\[
T=t_1V_{n(1)}\oplus t_2V_{n(2)}\oplus\cdots\in\Mcal,
\]
for a bounded sequence of $t_k\ge0$.
Then
\begin{gather}
|T|=t_1I_{n(1)}\oplus t_2I_{n(2)}\oplus\cdots \notag \\
\|T\|_p=\sum_{k=1}^\infty t_k^p\alpha(k). \label{eq:Tp}
\end{gather}
Taking the obvious diagonal spectral measure $E$ defined on the set
$\{(k,j)\in\Nats^2\mid j\le n(k)\}$ by
\[
E(\{(k,j)\})=
\underset{k-1\text{ times}}{\underbrace{0\oplus\cdots\oplus0}}
 \oplus e_{jj}\oplus0\oplus0\oplus\cdots
\]
and using the result of Lemma~\ref{lem:Unvar}, we find that the total variation
of the set function
\[
A_1\times\cdots\times A_p\mapsto\tau(E(A_1)TE(A_2)T\cdots E(A_p)T)
\]
is
\begin{equation}\label{eq:totvar}
\sum_{k=1}^\infty t_k^p\left(\frac{\alpha(k)}{n(k)}\right)n(k)^{p/2}
=\sum_{k=1}^\infty t_k^p\alpha(k)n(k)^{(p/2)-1}.
\end{equation}
Now assuming $n(k)$ is unbounded as $k\to\infty$ and fixing
an integer $p\ge3$, it is easy to choose values of $\alpha(k)$
and $t_k$ such that the $p$-norm in \eqref{eq:Tp} is finite
while the total variation~\eqref{eq:totvar} is infinite, in both cases (I) and (II) above.
\end{example}

\begin{remark}
The above examples also work to show that the set function
\[A_1\times\dots\times A_{p+1}\mapsto\tau[E(A_1)T\dots TE(A_{p+1})]\] has infinite total variation.
\end{remark}

\section{Main results}
\label{sec:mr}

In this section we state the main results which will be proved in the next three sections.

\begin{thm}\label{prop:tr}Let $2<p\in\Nats$.
Let $H_0$ be a self-adjoint operator affiliated with $\Mcal=\BH$ and $V$
a self-adjoint operator in $\ncs{2}$. Then, the following assertions
hold.\\
(i) There is a unique finite real-valued measure $\nu_p$ on $\Reals$ such that the trace formula
\begin{align}\label{f-la:tr}
\tau[R_{p,H_0,V}(f)]=\int_\Reals f^{(p)}(t)\,d\nu_p(t)
\end{align} holds for $f\in\mathcal{W}_p$. The total variation of $\nu_p$ is bounded by
\[\text{\rm var}(\nu_p)\leq c_p:=\frac{1}{p!}\norm{V}_2^p.\]
(ii) If, in addition, $H_0$ is bounded, then $\nu_p$ is absolutely continuous and supported in $[a,b]\supset\sigma(H_0)\cup\sigma(H_0+V)$.
The density $\eta_p$ of $\nu_p$ can be computed recursively by
\begin{align}\label{f-la:star}
\eta_{p}(t)=&\frac{\tau(V^{p-1})}{(p-1)!}-\nu_{p-1}((-\infty,t])\\\nonumber
&\quad-\frac{1}{(p-1)!}\int_{\Reals^{p-1}}\dd{\la_{p-1}}{p-2}{(\la-t)^{p-2}_+}\,
dm_{p-1,H_0,V}(\la_1,\dots,\la_{p-1}),
\end{align} for a.e. $t\in\Reals$. In this case, \eqref{f-la:tr} also holds for $f\in\Rfr$.
\end{thm}

\begin{thm}\label{prop:tr'}Let $p\in\{2,3\}$.
Let $H_0=H_0^*$ be an operator affiliated with a von Neumann algebra $\Mcal$ with normal faithful semi-finite trace $\tau$ and $V=V^*$ an operator in $\ncs{2}$. Then, the following assertions hold.\\
(i) There is a unique real-valued measure $\nu_p$ on $\Reals$ such that the trace formula \eqref{f-la:tr} holds for $f\in C_c^\infty(\Reals)$. If $H_0$ is bounded, then $\nu_p$ is finite and \eqref{f-la:tr} also holds for $f\in\mathcal{W}_p\cup\Rfr$.\\
(ii) The measure $\nu_2$ is absolutely continuous. If, in addition, $H_0$ is bounded, then $\nu_p$ is absolutely continuous for $p=3$.\\
(iii)Assume, in addition, that $V\in\ncs{1}$ if $p=2$ or that $H_0$ is bounded if $p=3$. Then the density $\eta_p$ of $\nu_p$ can be computed recursively by \eqref{f-la:star}.
\end{thm}

\begin{remark}\label{rem:see}
When $V\in\ncs{1}$, Koplienko's spectral shift function $\eta_2=\eta_{H_0,H_0+V}$ can be represented by \eqref{f-la:star}, which reduces to the known formula (see \cite{Kop84,convexity})
\begin{align*}
\eta_2(t)&=\tau(V)-\int_{-\infty}^t\xi_{H_0+V,H_0}(\la)\,d\la-\int_t^\infty d\tau[E_{H_0}(\la)V]\\
&=-\int_{-\infty}^t\xi_{H_0+V,H_0}(\la)\,d\la+\tau[E_{H_0}((-\infty,t))V].
\end{align*} For $V$ in the standard Hilbert-Schmidt class, no explicit formula for $\eta_{H_0,H_0+V}$ is known; existence of Koplienko's spectral shift function is proved implicitly by approximation of $V$ by finite-rank operators.
\end{remark}

\begin{remark}
Representation \eqref{f-la:koviaxi} of Koplienko's spectral shift function via Krein's spectral shift function was obtained by integrating by parts in the trace formula in \eqref{f-la:trp=1} \cite{Kop84,convexity}. When $V\in\ncs{1}$ and $f\in\Rfr_b$ (or $f\in\Rfr$ if $H_0$ is bounded), one can see from Lemma \ref{prop:res'} that $\tau\left(\frac{d}{dt}\big|_{t=0}f(H_0+tV)\right)=\tau[Vf'(H_0)]$, and thus,
\[\tau[R_{2,H_0,V}(f)]=\tau[f(H_0+V)-f(H_0)-Vf'(H_0)].\] When $\Mcal$ is finite and $H_0$ is bounded (so that $\eta_2$ is integrable and supported in a segment containing the spectra of $H_0$ and $H_0+V$), integrating by parts in Koplienko's trace formula in \eqref{f-la:trp=2} gives \begin{align}\label{f-la:itf}\nonumber
&\tau\left[f(H_0+V)-f(H_0)-Vf'(H_0)-\frac12V^2f''(H_0)\right]\\&=
\int_\Reals f'''(t)\left(-\int_{-\infty}^t\eta_2(\la)\,d\la+
\frac12\tau\big[V^2E_{H_0}((-\infty,t))\big]\right)\,dt.
\end{align} The bound for the remainder in the approximation formula \eqref{f-la:itf} is $\mathcal{O}\left(\norm{V}_2^2\right)$ since $\norm{\eta_2}_1=\frac{\norm{V}_2^2}{2}$ and $\eta_2\geq 0$ (see \cite{Kop84,convexity} for properties of $\eta_2$).
\end{remark}

\begin{cor}
Let $H_0$ be a self-adjoint operator in $\Mcal$ and $V$
a self-adjoint operator in $\ncs{p}$, where  $2<p\in\Nats$ if $\Mcal=\BH$ or $p=3$ if $\Mcal$ is a general semi-finite von Neumann algebra. Then, there exists a sequence $\{\eta_{p,n}\}_n$ of $L_\infty$-functions such that
\[\tau\big[R_{p,H_0,V}(f)\big]=\lim_{n\rightarrow\infty}\int_\Reals f^{(p)}(t)\eta_{p,n}(t)\,dt,\] for $f\in\mathcal{W}_p$.
\end{cor}

\begin{proof}
Given $V\in\ncs{p}$, there exists a sequence of operators $V_n\in\Mcal$, which are linear combinations of $\tau$-finite projections in $\Mcal$ (or just finite rank operators when $\Mcal=\BH$) such that $\lim_{n\rightarrow\infty}\norm{V-V_n}_p=0$. Then by Lemma \ref{prop:edertr} and Proposition \ref{prop:so*} for $f\in\mathcal{W}_p$,
\begin{align}
\label{f-la:p_1}
\lim_{n\rightarrow\infty}\tau[R_{p,H_0,V_n}(f)]=\tau[R_{p,H_0,V}(f)].
\end{align}
By Theorem \ref{prop:tr} in the case of $\Mcal=\BH$ or Theorem \ref{prop:tr'} in the case of a general $\Mcal$, respectively, applied to the $\tau$-Hilbert-Schmidt perturbations $V_n$, there exists a sequence $\{\eta_{p,n}\}_n$ of $L_\infty$-functions such that
\begin{align}
\label{f-la:p_2}
\tau\big[R_{p,H_0,V_n}(f)\big]=\int_\Reals f^{(p)}(t)\eta_{p,n}(t)\,dt.
\end{align}
Combining \eqref{f-la:p_1} and \eqref{f-la:p_2} completes the proof.
\end{proof}

\begin{thm}\label{prop:tr_free}
Let $\tau$ be finite and let $H_0=H_0^*$ be affiliated with the algebra $\Mcal$ and $V=V^*\in\Mcal$. Assume that $(zI-H_0)^{-1}$ and $V$ are free in the noncommutative space $(\Mcal,\tau)$. Then for $p\geq 3$ the following assertions hold.\\
(i) There is a unique finite real-valued measure $\nu_p$ on $\Reals$ such that the trace formula \eqref{f-la:tr} holds for $f\in\mathcal{W}_p$.\\
(ii) If, in addition, $H_0$ is bounded, then $\nu_p$ is absolutely continuous and supported in $[a,b]\supset\sigma(H_0)\cup\sigma(H_0+V)$. The density $\eta_p$ of $\nu_p$ can be computed recursively by \eqref{f-la:star}. In this case, \eqref{f-la:tr} also holds for $f\in\Rfr$.
\end{thm}

\section{Recursive formulas for the Cauchy transform}

\label{sec:ct}

Let $H_0$ and $V$ be self-adjoint operators in $\Mcal$. Assume, in
addition, that $V\in\ncs{2}$. In this section we investigate a measure $\nu_p=\nu_{p,H_0,V}$ as defined in \eqref{f-la:tr} for $f=f_z$ and $f(t)=t^p$. We derive properties of the Cauchy transform of the measure $\nu_p$ which will be used in section \ref{sec:bh} to show that the measure $\nu_{p+1}=\nu_{p+1,H_0,V}$ satisfying \eqref{f-la:tr} for $f\in\Rfr$ is absolutely continuous and its density can be determined explicitly via the density of $\nu_p$ and an integral of a spline function against a certain multiple spectral measure. In addition, for a general trace $\tau$, the results of this section will be used in section \ref{sec:vn} to prove existence of an absolutely continuous measure $\nu_3$ satisfying \eqref{f-la:tr}
for $p=3$ and find an explicit formula for the density of $\nu_3$.

Let $G_{\nu}$ denote the Cauchy transform of a finite measure $\nu$:
\begin{align}\label{f-la:g_def}
G_{\nu}(z)=\int_\Reals\frac{1}{z-t}\,d\nu(t),\quad \im(z)\neq
0.\end{align}

The goal of this section is to prove the theorem below.

\begin{thm}\label{prop:more}
Let $H_0=H_0^*\in\Mcal$ and $V=V^*\in\ncs{2}$. Suppose that $\nu_p$ is a real-valued absolutely continuous measure satisfying \eqref{f-la:tr} for $f=f_z$ and $f(t)=t^p$. Let $G:\Complex_+\to\Complex$ be the analytic function satisfying
\begin{align}\label{f-la:ctd}
&G^{(p+1)}(z)=-G_{\nu_p}^{(p)}(z)
-(-1)^{(p+1)}\tau\big[\big((zI-H_0)^{-1}V\big)^p(zI-H_0)^{-1}\big],\\ \label{f-la:cta}
& \lim_{\im|z|\rightarrow\infty}G(z)=0.
\end{align}Then $G(z)$ is the Cauchy transform of the measure $\nu_{p+1}$ satisfying \eqref{f-la:tr} for $f=f_z$, which is absolutely continuous with the density given by \[\eta_{p+1}(t)=\frac{1}{p!}\left(\tau(V^p)-p!\nu_p((-\infty,t])
-\int_{\Reals^p}\dd{\la_p}{p-1}{(\la-t)^{p-1}_+}\,dm_{p,H_0,V}(\la_1,\dots,\la_p)\right),\]
for a.e. $t\in\Reals$.
\end{thm}

\begin{lemma}\label{prop:int_m}Let $\nu_p$ be a measure satisfying \eqref{f-la:tr} for $f(t)=t^p$. Then
\[\int_\Reals d\nu_p(t)=\frac{1}{p!}\tau(V^p).\]
\end{lemma}

\begin{proof}Applying the trace formula \eqref{f-la:tr} to the
polynomial $f(t)=t^p$ and applying Lemma \ref{prop:res'} give
\begin{align*}
p!\int_\Reals
d\nu_p(t)=\tau\left[(H_0+V)^p-\sum_{j=0}^{p-1}\sum_{\substack{k_0,k_1,\dots,k_j\geq
0\\ k_0+k_1+\dots+k_j=p-j}}H_0^{p_0}VH_0^{p_1}V\dots
VH_0^{p_j}\right]=\tau(V^p).
\end{align*}
\end{proof}

\begin{lemma}\label{prop:g_der-s}Let $H_0=H_0^*\in\Mcal$ and $V=V^*\in\ncs{p}$. Let $\nu_p$ and $\nu_{p+1}$ be compactly supported measures. Then $\nu_p$ and $\nu_{p+1}$ satisfy \eqref{f-la:tr} for $f=f_z$ if and only if
\[G_{\nu_{p+1}}^{(p+1)}(z)=-G_{\nu_p}^{(p)}(z)
-(-1)^{(p+1)}\tau\big[\big((zI-H_0)^{-1}V\big)^p(zI-H_0)^{-1}\big].\]
\end{lemma}

\begin{proof}The result follows immediately from Lemma
\ref{prop:R_rec} upon employing the straightforward equality
\[(-1)^pG_{\nu_p}^{(p)}(z)=\tau\left[R_{p,H_0,V}\left(f_z\right)\right].\]
\end{proof}

Lemma \ref{prop:g_der-s} will be used to construct an absolutely continuous measure $\nu_{p+1}$ satisfying \eqref{f-la:tr} for $f=f_z$ based on the existence of an absolutely continuous measure $\nu_p$ satisfying \eqref{f-la:tr} for $f=f_z$.

\begin{lemma}\label{prop:g_int}
Let $H_0=H_0^*\in\Mcal$ and $V=V^*\in\ncs{p}$. Let $\nu_p$ be a measure satisfying \eqref{f-la:tr} for $f=f_z$ and $f(t)=t^p$. Assume, in addition, that $\nu_p$ is absolutely continuous with the density $\eta_p$ compactly supported in $[a,b]$. Assume that $G: \Complex_+\to\Complex$ is an analytic function satisfying \eqref{f-la:ctd}. Then $G$ is determined by
\begin{align*}
G(z)=&-\log(z-b)\frac{1}{p!}\tau(V^p)-
\int_\Reals\frac{1}{z-\la}\chi_{[a,b]}(\la)\int_a^\la\eta_p(t)\,dt\,d\la\\
&+(-1)^{(p+1)}\frac{1}{p}\int\dots\int\tau\big[\big((zI-H_0)^{-1}V\big)^p\big]\,dz^p,
\end{align*} up to a polynomial of degree $p$.
\end{lemma}

\begin{proof}We note that
\[-\tau\big[\big((zI-H_0)^{-1}V\big)^p(zI-H_0)^{-1}\big]=
\frac{d}{dz}\left(\frac{1}{p}\tau\big[\big((zI-H_0)^{-1}V\big)^p\big]\right).\]
Then by Lemma \ref{prop:g_der-s},
\begin{align}\label{f-la:parts0}
G(z)=-\int G_{\nu_p}(z)\,dz
+(-1)^{(p+1)}\frac{1}{p}\int\dots\int\tau\big[\big((zI-H_0)^{-1}V\big)^p\big]\,dz^p.
\end{align}
By the assumption of the lemma, $d\nu_p(\la)=\eta_p(\la)\,d\la$, and hence,
\begin{align*}
G_{\nu_p}(z)=\int_a^b\frac{1}{z-\la}\eta_p(\la)\,d\la.
\end{align*}Integrating the latter expression by parts gives
\begin{align}\label{f-la:parts1}
G_{\nu_p}(z)=\left(\frac{1}{z-\la}\int_a^\la\eta_p(t)\,dt\right)\bigg|_a^b
-\int_a^b\frac{1}{(z-\la)^2}\int_a^\la\eta_p(t)\,dt\,d\la.
\end{align}By Lemma \ref{prop:int_m}, the first summand in \eqref{f-la:parts1} equals
\begin{align}
\label{f-la:parts2}
\frac{1}{z-b}\int_a^b\eta_p(t)\,dt=\frac{1}{z-b}\frac{1}{p!}\tau(V^p).
\end{align} The second summand in \eqref{f-la:parts1} equals
\begin{align}
\label{f-la:parts3}
\frac{d}{dz}\left(\int_a^b\frac{1}{z-\la}\int_a^\la\eta_p(t)\,dt\,d\la\right).
\end{align}
Combining \eqref{f-la:parts0}-\eqref{f-la:parts3} completes the proof.
\end{proof}

\begin{thm}\label{prop:G_rec}
Let $H_0=H_0^*\in\Mcal$ and $V=V^*\in\ncs{2}$. Let $[a,b]$ be a segment containing $\sigma(H_0)\cup\sigma(H_0+V)$. Assume that either $\tau$ is standard or $p=2$. Let $\nu_p$ be a measure compactly supported in $[a,b]$ and satisfying \eqref{f-la:tr} for $f=f_z$ and $f(t)=t^p$. Assume, in addition, that $\nu_p$ is absolutely continuous with the density $\eta_p$. Then the function
\begin{align*}
G(z)=
\frac{1}{p!}\int_\Reals\frac{1}{z-t}&\bigg(\tau(V^p)-p!\nu_p((-\infty,t])\\\nonumber
&-\int_{\Reals^p}\dd{\la_p}{p-1}{(\la-t)^{p-1}_+}\,dm_{p,H_0,V}(\la_1,\dots,\la_p)\bigg)\,dt
\end{align*} satisfies \eqref{f-la:ctd} and \eqref{f-la:cta}.
\end{thm}

\begin{proof}Since $d\nu_p(t)=\eta_p(t)\,dt$, we have
\begin{align}
\label{f-la:k-1}
\chi_{[a,b]}(\la)\int_a^\la\eta_p(t)\,dt=\nu_p((-\infty,\la])\chi_{[a,b]}(\la).
\end{align}
By Remark \ref{prop:imr}, we obtain the representation
\begin{align}\label{f-la:k0}
\tau\left[\left((zI-H_0)^{-1}V\right)^p\right]=
\int_{\Reals^p}\dd{\la_p}{p-1}{\frac{1}{z-\la}}\,dm_{p,H_0,V}(\la_1,\dots,\la_p).
\end{align}Since $\sigma(H_0)\cup\sigma(H_0+V)\subset[a,b]$, the measure $m_{p,H_0,V}$ is supported in $[a,b]$. By Lemma \ref{prop:dd_der} (i), we can interchange the order of integration in
\begin{align*}
&\int\dots\int\tau\big[\big((zI-H_0)^{-1}V\big)^p\big]\,dz^p\\
&\quad=\int\dots\int\left(\int_{[a,b]^p}\dd{\la_p}{p-1}{\frac{1}{z-\la}}\,
dm_{p,H_0,V}(\la_1,\dots,\la_p)\right)\,dz^p
\end{align*}
and obtain
\begin{align}\label{f-la:k1}
&\int\dots\int\tau\big[\big((zI-H_0)^{-1}V\big)^p\big]\,dz^p\\\nonumber
&=\int_{[a,b]^p}\dd{\la_p}{p-1}{\int\dots\int\frac{1}{z-\la}\,dz^p}\,
dm_{p,H_0,V}(\la_1,\dots,\la_p),
\end{align} with a suitable choice of constants of integration on the left-hand side of \eqref{f-la:k1}. For a reason to become clear later, we choose the antiderivatives in \eqref{f-la:k1} with real constants of integration.
Since
\[\int\dots\int\frac{1}{z-\la}\,dz^p
=(z-\la)^{p-1}\log(z-\la)+\alpha_{p-1}z^{p-1}+pol_{p-2}(z),\] with $pol_{p-2}(z)$ a polynomial of degree $p-2$ and a constant $\alpha_{p-1}\in\Reals$ to be fixed later, we obtain by Proposition \ref{prop:dds} \eqref{f-la:dd2} that the expression in \eqref{f-la:k1} equals
\begin{align}\label{f-la:k1'}
\frac{1}{(p-1)!}
\int_{[a,b]^p}\left(\dd{\la_p}{p-1}{(z-\la)^{p-1}\log(z-\la)}+\alpha_{p-1}\right)\,
dm_{p,H_0,V}(\la_1,\dots,\la_p).
\end{align}
By Lemma \ref{prop:g_int} and \eqref{f-la:k-1} - \eqref{f-la:k1'},
\begin{align}\label{f-la:k2}
G(z)=&-\int_a^b\frac{1}{z-t}\nu_p((-\infty,t])\,dt\\\nonumber
&+\frac{(-1)^{p+1}}{p!}\int_{[a,b]^p}
\bigg(\dd{\la_p}{p-1}{(z-\la)^{p-1}\log(z-\la)}\\
\nonumber&\quad\quad\quad\quad\quad\quad+(-1)^{p}\log(z-b)+\alpha_{p-1}\bigg)\,
dm_{p,H_0,V}(\la_1,\dots,\la_p).
\end{align}

Now we will represent the second integral in \eqref{f-la:k2} as the Cauchy transform of an absolutely continuous measure. If not all $\la_1,\la_2,\dots,\la_p$ coincide, then by Proposition \ref{prop:dds} \eqref{f-la:dd4} and \eqref{f-la:dd5},
\begin{align}\label{f-la:k2a}\nonumber
&\dd{\la_p}{p-1}{(z-\la)^{p-1}\log(z-\la)}\\\nonumber
&\quad=\frac{1}{(p-2)!}\int_\Reals \frac{\partial^{p-1}}{\partial t^{p-1}}\big((z-t)^{p-1}\log(z-t)\big)\dd{\la_p}{p-1}{(\la-t)_+^{p-2}}\,dt\\
\nonumber
&\quad=\frac{1}{(p-2)!}\int_\Reals\big((-1)^{p-1}(p-1)!\log(z-t)+
\gamma_{p-1}\big)\dd{\la_p}{p-1}{(\la-t)_+^{p-2}}\,dt\\
&\quad=(-1)^{p-1}(p-1)\int_\Reals\log(z-t)\dd{\la_p}{p-1}{(\la-t)_+^{p-2}}\,dt+
\frac{1}{(p-1)!}\gamma_{p-1},
\end{align} with $\gamma_{p-1}\in\Reals$. By \eqref{f-la:k2a} and Proposition \ref{prop:dds} \eqref{f-la:dd5}, we obtain
\begin{align}
\label{f-la:k2b}\nonumber
&J_{\la_1,\dots,\la_p}(z)=\dd{\la_p}{p-1}{(z-\la)^{p-1}\log(z-\la)}+(-1)^{p}\log(z-b)+\alpha_{p-1}\\
&\quad=(-1)^{p-1}(p-1)\int_\Reals\big(\log(z-t)-\log(z-b)\big)
\dd{\la_p}{p-1}{(\la-t)_+^{p-2}}\,dt\\\nonumber
&\quad\quad+\frac{1}{(p-1)!}\gamma_{p-1}+\alpha_{p-1}.
\end{align} Since in \eqref{f-la:k2} we need only to consider $\la_1,\dots,\la_p\in (a,b)$ and $\dd{\la_p}{p-1}{(\la-t)_+^{p-2}}$ is supported in
$[\min\{\la_1,\dots,\la_p\},\max\{\la_1,\dots,\la_p\}]$, we obtain that in \eqref{f-la:k2b} it is enough to take $t\in [a,b]$.
By standard computations, for $t<b$,
\begin{align}\label{f-la:lc+}
\text{the function }\;
z\mapsto\log(z-t)-\log(z-b)\;\text{ maps }\Complex_+ \text{ to }\Complex_-
\end{align} and
\begin{align}\label{f-la:lcl}
\lim_{y\rightarrow\infty}\i y\big(\log(\i y-t)-\log(\i y-b)\big)=b-t.
\end{align} Let $\alpha_{p-1}=-\frac{1}{(p-1)!}\gamma_{p-1}$. Then \eqref{f-la:lc+} and \eqref{f-la:lcl} along with Proposition \ref{prop:dds} \eqref{f-la:dd5} imply that $J_{\la_1,\dots,\la_p}$ in \eqref{f-la:k2b} maps $\Complex_+$ to $\Complex_\pm$ (depending on the sign of $(-1)^{p-1}$) and
$\lim_{y\rightarrow\infty}\i y J_{\la_1,\dots,\la_p}(\i y)\in\Reals$.
By the classical theory of analytic functions, $J_{\la_1,\dots,\la_p}$ is the Cauchy transform of a finite real-valued measure.
If $\la_1=\la_2=\dots=\la_p$, then
\begin{align}
\label{f-la:k2c}\nonumber
&J_{\la_1,\dots,\la_1}(z)
=\dd{\la_p}{p-1}{(z-\la)^{p-1}\log(z-\la)}+(-1)^{p}\log(z-b)+\alpha_{p-1}\\
&\quad=(-1)^{p-1}\big(\log(z-\la_1)-\log(z-b)\big)+\alpha_{p-1}.
\end{align} By \eqref{f-la:lc+} and \eqref{f-la:lcl}, the function $J_{\la_1,\dots,\la_1}$ is also the Cauchy transform of a finite real-valued measure. Below we show that the measure generating $J_{\la_1,\dots,\la_p}$ is absolutely continuous.

If all $\la_1,\la_2,\dots,\la_p$ are distinct, then by Proposition \ref{prop:dds} \eqref{f-la:dd7},
\begin{align*}
\dd{\la_p}{p-1}{(z-\la)^{p-1}\log(z-\la)}=
\sum_{k=1}^{p}\frac{(z-\la_k)^{p-1}\log(z-\la_k)}{\prod_{j\neq k}(\la_k-\la_j)}.
\end{align*}Since $\dd{\la_p}{p-1}{(z-\la)^{p-1}\log(z-\la)}$ is symmetric in $\la_1,\la_2,\dots,\la_p$, we may assume without loss of generality that $\la_1<\la_2<\dots<\la_p$. Then
\begin{align}
\label{f-la:k3}\nonumber
\phi(t):&=-\frac{1}{\pi}\lim_{\varepsilon\rightarrow 0^+}\im\big((-1)^{p}\big(\log(t+\i\varepsilon-b)+\alpha_{p-1}\big)\big)\\
&\quad-\frac{1}{\pi}\lim_{\varepsilon\rightarrow 0^+}\im
\big(\dd{\la_p}{p-1}{(t+\i\varepsilon-\la)^{p-1}\log(t+\i\varepsilon-\la)}\big)\\ \nonumber
&=\begin{cases}
(-1)^{p+1}+(-1)^p\sum_{k=1}^{p}\frac{(\la_k-t)^{p-1}}{\prod_{j\neq k}(\la_k-\la_j)} & \text{ if }t<\la_1\\[1.5ex]
(-1)^{p+1}+(-1)^p\sum_{k=m}^{p}\frac{(\la_k-t)^{p-1}}{\prod_{j\neq k}(\la_k-\la_j)} & \text{ if }\la_{m-1}\leq t<\la_m,\text { for } 2\leq m\leq p\\[1.5ex]
(-1)^{p+1} & \text{ if }\la_p\leq t<b\\[1.5ex]
0 &\text{ if }t\geq b.
\end{cases}
\end{align}
By Proposition \ref{prop:dds} \eqref{f-la:dd7} and \eqref{f-la:dd2},
\begin{align}\label{f-la:k4}
(-1)^{p+1}+(-1)^p\sum_{k=1}^{p}\frac{(\la_k-t)^{p-1}}{\prod_{j\neq k}(\la_k-\la_j)}
=(-1)^{p+1}+(-1)^p\dd{\la_p}{p-1}{(\la-t)^{p-1}}=0,
\end{align}and hence, $\phi$ is supported in $[a,b]$.
Combining \eqref{f-la:k3} and \eqref{f-la:k4} gives
\begin{align}\label{f-la:k5}
\phi(t)=(-1)^{p+1}\chi_{(-\infty,b]}(t)+(-1)^p\dd{\la_p}{p-1}{(\la-t)_+^{p-1}}.
\end{align}
Similarly, with the use of Definition \ref{prop:dddef} and Lemma \ref{prop:dd_der} (ii), one can see that \eqref{f-la:k5} holds when some of the values $\la_1,\la_2,\dots,\la_p$ repeat.
Combining \eqref{f-la:k2} and \eqref{f-la:k5} gives
\begin{align}\label{f-la:k6a}
G(z)=&-\int_a^b\frac{1}{z-t}\nu_p((-\infty,t])\,dt\\\nonumber
&+\frac{1}{p!}\int_{[a,b]^p}\int_\Reals\frac{1}{z-t}
\left(\chi_{(-\infty,b]}(t)-\dd{\la_p}{p-1}{(\la-t)_+^{p-1}}\right)\,dt\,
dm_{p,H_0,V}(\la_1,\dots,\la_p).
\end{align}
Changing the order of integration in the second integral in \eqref{f-la:k6a} and applying
Lemma \ref{prop:int_m} along with the fact that $\nu_p$ is supported in $[a,b]$ imply the representation
\begin{align*}
G(z)=
&\int_\Reals\frac{1}{z-t}\bigg(\chi_{(-\infty,b]}(t)
\left(\frac{\tau(V^p)}{p!}-\nu_p((-\infty,t])\right)\\\nonumber
&\quad-\frac{1}{p!}\int_{\Reals^p}\dd{\la_p}{p-1}{(\la-t)_+^{p-1}}\,
dm_{p,H_0,V}(\la_1,\dots,\la_p)\bigg)\,dt\\\nonumber
=&\int_\Reals\frac{1}{z-t}\bigg(
\frac{\tau(V^p)}{p!}-\nu_p((-\infty,t])\\\nonumber
&\quad-\frac{1}{p!}\int_{\Reals^p}\dd{\la_p}{p-1}{(\la-t)_+^{p-1}}\,
dm_{p,H_0,V}(\la_1,\dots,\la_p)\bigg)\,dt.
\end{align*}
\end{proof}

\begin{proof}[Proof of Theorem \ref{prop:more}] In view of Theorem \ref{prop:G_rec}, it is enough to prove that the function
\begin{align}\label{f-la:ddm0}
t\mapsto\frac{1}{p!}\left(\tau(V^p)-p!\nu_p((-\infty,t])
-\int_{\Reals^p}\dd{\la_p}{p-1}{(\la-t)^{p-1}_+}\,dm_{p,H_0,V}(\la_1,\dots,\la_p)\right)
\end{align} is real-valued.
The integral
\begin{align}\label{f-la:ddm}
\int_{\Reals^p}\dd{\la_p}{p-1}{(\la-t)^{p-1}_+}\,dm_{p,H_0,V}(\la_1,\dots,\la_p)
\end{align}can be written as
\begin{align}\label{f-la:re+im}
&\int_{\Reals^p}\dd{\la_p}{p-1}{(\la-t)^{p-1}_+}\,d\re \big(m_{p,H_0,V}(\la_1,\dots,\la_p)\big)\\\nonumber
&\quad+\i\int_{\Reals^p}\dd{\la_p}{p-1}{(\la-t)^{p-1}_+}\,d\im \big(m_{p,H_0,V}(\la_1,\dots,\la_p)\big).
\end{align}
It is easy to see that
\[\overline{m_{p,H_0,V}(d\la_1,d\la_2,\dots,d\la_{p-1},d\la_p)}=
m_{p,H_0,V}(d\la_p,d\la_{p-1},\dots,d\la_2,d\la_1),\] and hence,
\begin{align}\label{f-la:impem}
\im \big(m_{p,H_0,V}(d\la_1,d\la_2,\dots,d\la_{p-1},d\la_p)\big)
=-\im \big(m_{p,H_0,V}(d\la_p,d\la_{p-1},\dots,d\la_2,d\la_1)\big).
\end{align} Along with symmetry of the divided difference $\dd{\la_p}{p-1}{(\la-t)^{p-1}_+}$ in $\la_1,\dots,\la_p$, the equality \eqref{f-la:impem} implies that the second integral in \eqref{f-la:re+im} equals $0$, and thus \eqref{f-la:ddm} is real-valued. We have that $\nu_1$ and $\eta_1$ are real-valued. By induction, we obtain that $\nu_p$ and $\eta_p$ are real-valued for every $p\in\Nats$. Therefore, \eqref{f-la:ddm0} is real-valued.
\end{proof}



\section{Spectral shift functions for $\Mcal=\BH$}

\label{sec:bh}

\begin{proof}[Proof of Theorem \ref{prop:tr} (i)] Let $H_x=H_0+xV$.
The proof of the theorem will proceed in several steps.

{\it Step 1.} Assume first that $H_0$ is bounded and $f\in\Rfr$. Let $[a,b]$ be a segment containing $\sigma(H_0)\cup\sigma(H_0+V)$. By Corollary
\ref{prop:m.k}, the finitely additive measure defined on rectangles by
\[m_{p,H_x,V}^{(1)}(A_1\times A_2\times\cdots\times
A_p\times A_{p+1})=\tau\big[E_{H_x}(A_1)VE_{H_x}(A_2)V\dots
E_{H_x}(A_p)VE_{H_x}(A_{p+1})\big],\] with $A_1,\ldots,A_{p+1}$ Borel subsets of
$\Reals$, extends to a countably additive measure with total
variation not exceeding $\norm{V}_2^p$. It follows from Corollary \ref{prop:i_to_m} and Remark \ref{prop:imrp} that
\begin{align}\label{f-la:derdif}
\tau\left[\frac{d^p}{dx^p}f(H_0+xV)\right]=p!\int_{\Reals^{p+1}}
\dd{\la_{p+1}}{p}{f}\,dm_{p,H_x,V}^{(1)}(\la_1,\la_2,\dots,\la_{p+1}).\end{align}
By Proposition \ref{prop:dds} \eqref{f-la:dd6},
\[\big|\dd{\la_{p+1}}{p}{f}\big|\leq\frac{1}{p!}\max_{\la\in [a,b]}|f^{(p)}(\la)|,\]
which along with \eqref{f-la:derdif} ensures that
\[\left|\tau\left[\frac{d^p}{dx^p}f(H_0+xV)\right]\right|\leq\norm{V}_2^p\max_{\la\in [a,b]}|f^{(p)}(\la)|.\]
Applying the latter estimate to the integrand in \eqref{f-la:Rpf}
guarantees that $R_{p,H_0,V}(f)$ is a bounded functional on the space of $f^{(p)}$
with the norm not exceeding $\frac{1}{p!}\norm{V}_2^p$. Therefore, there
exists a measure $\nu_{p,H_0,V}$ supported in $[a,b]$ and of variation not exceeding $\frac{1}{p!}\norm{V}_2^p$ such that
\begin{align}\label{f-la:tr.bdd}\tau[R_{p,H_0,V}(f)]=\int_a^b
f^{(p)}(t)\,d\nu_{p,H_0,V}(t),\end{align} for all $f\in\Rfr$.

{\it Step 2.} We prove the claim of the theorem for $H_0$ bounded and $f\in\mathcal{W}_p$. Repeating the reasoning
of \cite[Theorem 2.8]{Dostanic}, one extends \eqref{f-la:tr.bdd} from
$\Rfr$ to the set of functions $\Reals\ni\la\mapsto e^{\i t\la}$,
$t\in\Reals$, as follows. By Runge's Theorem, there exists a
sequence of rational functions $r_n$ with poles off $D=\{\la\st
|\la|\leq 1+\norm{H_0}+\norm{V}\}$ such that
\[r_n^{(k)}(\la)\rightarrow (\i t)^k e^{\i t\la},\quad \la\in D,\quad k=0,1,2,\dots,\]
where the convergence is understood in the uniform sense. Making use
of Lemma \ref{prop:fcalc} and passing to the limit on both sides of
\eqref{f-la:tr.bdd} written for $f\in\Rfr$ proves \eqref{f-la:tr.bdd} for $f(\la)=e^{\i
t\la}$, with the same measure $\nu_{p,H_0,V}$ as at the previous step. Finally,
applying Corollary \ref{prop:dos_der} extends \eqref{f-la:tr.bdd} to the
class of  $f\in\mathcal{W}_p$, with the same measure $\nu_{p,H_0,V}$.

{\it Step 3.} Now we extend \eqref{f-la:tr} to the case of an
unbounded operator $H_0$ and $f\in\mathcal{W}_p$. This is done similarly to \cite[Lemma 2.7]{Dostanic}, with replacement of iterated operator integrals by multiple operator integrals. Let $H_{0,n}=E_{H_0}((-n,n))H_0$ and $H_{x,n}=H_{0,n}+xV$.
It follows from \eqref{f-la:Rpf} of Theorem \ref{prop:Rpf} that
\begin{align*}&
R_{p,H_0,V}(f)-R_{p,H_{0,n},V}(f)\\&\quad=\frac{1}{(p-1)!}\int_0^1\left(\frac{d^p}{dx^p}f(H_x)- \frac{d^p}{dx^p}
f(H_{x,n})\right)(1-x)^{p-1}\,dx.\end{align*} There exists a finite Borel measure  $\mu_f$ such that $f(\la)=\int_\Reals e^{it\la}\,d\mu_f(t)$.
On the strength of Lemma \ref{prop:edertr},
\begin{align}\label{f-la:elim0}
&\tau\left[\frac{d^p}{dx^p}f(H_x)-\frac{d^p}{dx^p}f(H_{x,n})\right]\\\nonumber
&=p!\int_{\Pi^{(p)}}\tau\left[e^{\i
(s_0-s_1)H_x}V\dots V e^{\i s_p H_x}-e^{\i
(s_0-s_1)H_{x,n}}V\dots V e^{\i s_p H_{x,n}}\right]\,d\sigma_f^{(p)}(s_0,\dots,s_p).
\end{align}
Proposition \ref{prop:so*} implies that the integrand in \eqref{f-la:elim0} converges to $0$, and hence, the whole expression in \eqref{f-la:elim0} converges to $0$ as $n\rightarrow\infty$. Then applying Proposition \ref{prop:so*} yields
\begin{align}\label{f-la:Helly}\nonumber
&\lim_{n\rightarrow\infty}\tau\left[R_{p,H_0,V}(f)-R_{p,H_{0,n},V}(f)\right]\\
&\quad=\lim_{n\rightarrow\infty}\frac{1}{(p-1)!}\int_0^1\tau\left[\frac{d^p}{dx^p}f(H_x)- \frac{d^p}{dx^p}
f(H_{x,n})\right](1-x)^{p-1}\,dx
=0.
\end{align}By the result of the previous step applied to the bounded operators
$H_{0,n}$, there is a sequence of measures $\nu_{p,H_{0,n},V}$ of variation bounded by $c_p$,
representing the functionals $R_{p,H_{0,n},V}(f)$ for $f\in\mathcal{W}_p$. Denote by $F_n$ the distribution function of $\nu_{p,H_{0,n},V}$. By Helly's selection theorem, there is a subsequence $\{F_{n_k}\}_k$ and a function $F$ of variation not exceeding $c_p$ such that $F_{n_k}$ converges to $F$ pointwise and in $L^1_{loc}(\Reals)$. The trace formula \eqref{f-la:tr} for bounded operators and the convergence in \eqref{f-la:Helly} ensure that the measure with the distribution $F$ satisfies \eqref{f-la:tr} for $f\in\mathcal{W}_p$.
\end{proof}

\begin{proof}[Proof of Theorem \ref{prop:tr} (ii)]
It is an immediate consequence of Theorem \ref{prop:tr} (i) and Theorem \ref{prop:more}.
\end{proof}

\section{Spectral shift functions for an arbitrary semi-finite $\Mcal$}

\label{sec:vn}


\begin{proof}[Proof of Theorem \ref{prop:tr'} for $p=2$]
Due to Theorem \ref{prop:m} and Corollary \ref{prop:i_to_m}, the proof of existence of Koplienko's spectral shift function $\eta_2$ for a Hilbert-Schmidt perturbation $V$ \cite[Lemma 3.3]{Kop84} (cf. also \cite{BoucletKo}) can be extended to the case of a $\tau$-Hilbert-Schmidt perturbation.
\end{proof}

The proof of Theorem \ref{prop:tr'} for $p=3$ will be based on the fact (see the lemma below) that if a measure (possibly complex-valued) satisfies \eqref{f-la:tr} for
$f=f_z$, then it satisfies \eqref{f-la:tr} for any $f\in\Rfr_b$.

\begin{lemma}\label{prop:G_extend} Let $H_0=H_0^*$ be an operator affiliated
with $\Mcal$ and $V=V^*\in\ncs{2}$. Let $\nu_p$, with $p=3$, be a Borel measure
satisfying
\begin{align*}
R_{p,H_0,V}\left(f_z\right)=p!\int_\Reals\frac{1}{(z-t)^{p+1}}\,d\nu_p(t).
\end{align*} Then, for all $f\in\Rfr_b$,
\begin{align}\label{f-la:tr'}
R_{p,H_0,V}(f)=\int_\Reals f^{(p)}(t)\,d\nu_p(t).
\end{align}
If, in addition, $H_0$ is bounded and $\nu_p$ is compactly supported,
then \eqref{f-la:tr'} holds for $f\in\Rfr$.
\end{lemma}

To prove Lemma \ref{prop:G_extend}, we need a simple lemma below.

\begin{lemma}\label{prop:g_der}Assume that the trace formula \eqref{f-la:tr} holds for $f=f_z$ with a finite measure $\nu_p$. Then,
\begin{align}\label{f-la:g_der1}
G_{\nu_p}^{(p)}(z)&=(-1)^p\tau\left[(zI-H_0-V)^{-1}-
\sum_{j=0}^{p-1}(zI-H_0)^{-1}\left(V(zI-H_0)^{-1}\right)^j\right]\\\label{f-la:g_der2}
&=(-1)^p\tau\big[(zI-H_0-V)^{-1}\left(V(zI-H_0)^{-1}\right)^p\big].
\end{align}
\end{lemma}

\begin{proof}
Differentiating the integral in \eqref{f-la:g_def} gives
\begin{align}\label{f-la:g'}
G_{\nu_p}^{(p)}(z)=(-1)^p
p!\int_\Reals\frac{1}{(z-t)^{p+1}}\,d\nu_p(t),\quad \im(z)\neq 0.\end{align}
Applying the trace formula \eqref{f-la:tr} to $f=f_z$ ensures
\begin{align}\label{f-la:g''}
\tau\left[(zI-H_0-V)^{-1}-
\sum_{j=0}^{p-1}(zI-H_0)^{-1}\left(V(zI-H_0)^{-1}\right)^j\right]=
p!\int_\Reals\frac{1}{(z-t)^{p+1}}\,d\nu_p(t).
\end{align}
Comparing \eqref{f-la:g''} with \eqref{f-la:g'} completes the proof of
\eqref{f-la:g_der1}; comparing \eqref{f-la:g''} with \eqref{f-la:g_Rres2} of Lemma \ref{prop:res} completes the proof of \eqref{f-la:g_der2}.
\end{proof}

\begin{proof}[Proof of Lemma \ref{prop:G_extend}]
{\it Step 1.} Assume that $H_0$ is bounded. We prove the claim for $f$ a polynomial. For
$z\in\Complex\setminus\Reals$, with $|z|$ large enough,
\[G_{\nu_p}(z)=\sum_{k=0}^\infty z^{-(k+1)}\int_\Reals t^k\,d\nu_p(t),\] and hence,
\begin{align}
\label{f-la:gp1} (-1)^p G_{\nu_p}^{(p)}(z)=\sum_{k=0}^\infty
z^{-(k+p+1)}(k+1)(k+2)\dots (k+p)\int_\Reals t^k\,d\nu_p(t).\end{align} On
the other hand,
\begin{align}\label{f-la:gp2}\nonumber
&(-1)^p G_{\nu_p}^{(p)}(z)\\\nonumber&\quad=
\tau\left[(zI-H_0-V)^{-1}-\sum_{j=0}^{p-1}(zI-H_0)^{-1}\left(V(zI-H_0)^{-1}\right)^j\right]\\
&\quad= \tau\left[\frac{1}{z}\left(I-\frac{H_0+V}{z}\right)^{-1}-
\sum_{j=0}^{p-1}\frac{1}{z^{j+1}}\left(I-\frac{H_0}{z}\right)^{-1}
\left(V\left(I-\frac{H_0}{z}\right)^{-1}\right)^j\right].
\end{align}
Employing the power series expansion in \eqref{f-la:gp2} gives
\begin{align}\label{f-la:gp2'}\nonumber
&(-1)^p G_{\nu_p}^{(p)}(z)\\\nonumber&=
\tau\bigg[\frac{1}{z}\sum_{m=0}^\infty\left(\frac{H_0+V}{z}\right)^{m}
-\sum_{j=0}^{p-1}\sum_{i=0}^\infty\frac{1}{z^{j+1}}
\sum_{\substack{k_0,k_1,\dots,k_j\geq 0\\
k_0+k_1+\dots+k_j=i}}\left(\frac{H_0}{z}\right)^{k_0}
V\left(\frac{H_0}{z}\right)^{k_1}V\dots
V\left(\frac{H_0}{z}\right)^{k_j}\bigg]\\
&=\tau\bigg[\sum_{m=0}^\infty z^{-(m+1)}(H_0+V)^{m}-\sum_{j=0}^{p-1}\sum_{i=0}^\infty z^{-(j+1)}
\sum_{\substack{k_0,k_1,\dots,k_j\geq 0\\k_0+k_1+\dots+k_j=i}}z^{-i}H_0^{k_0}
VH_0^{k_1}V\dots VH_0^{k_j}\bigg].
\end{align}
By expanding $(H_0+V)^m$ one can see that
\begin{align}\label{f-la:gp3}
&\tau\bigg[\sum_{m=0}^{p-1} z^{-(m+1)}(H_0+V)^{m}
-\sum_{j=0}^{p-1}\sum_{i=0}^{p-1-j} z^{-(j+1)}
\sum_{\substack{k_0,k_1,\dots,k_j\geq 0\\ k_0+k_1+\dots+k_j=i}}z^{-i}H_0^{k_0}
VH_0^{k_1}V\dots VH_0^{k_j}\bigg]=0.
\end{align}
Subtracting \eqref{f-la:gp3} from \eqref{f-la:gp2'} yields
\begin{align}\label{f-la:gp4}\nonumber
&(-1)^p G_{\nu_p}^{(p)}(z)\\\nonumber&=\tau\bigg[\sum_{m=p}^\infty
z^{-(m+1)}(H_0+V)^{m}-\sum_{j=0}^{p-1}\sum_{i=p-j}^\infty z^{-(i+j+1)}
\sum_{\substack{k_0,k_1,\dots,k_j\geq 0\\ k_0+k_1+\dots+k_j=i}}H_0^{k_0}
VH_0^{k_1}V\dots
VH_0^{k_j}\bigg]\\&=\tau\bigg[\sum_{m=p}^\infty
z^{-(m+1)}\bigg((H_0+V)^{m}-\sum_{j=0}^{p-1}\sum_{\substack{k_0,k_1,\dots,k_j\geq
0\\ k_0+k_1+\dots+k_j=m-j}}H_0^{k_0} VH_0^{k_1}V\dots VH_0^{k_j}\bigg)\bigg].
\end{align}
By the continuity of the trace $\tau$, \eqref{f-la:gp4} can be
rewritten as
\begin{align}\label{f-la:gp4'}\nonumber
&(-1)^p G_{\nu_p}^{(p)}(z)\\\nonumber&=\sum_{m=p}^\infty
z^{-(m+1)}\tau\bigg[(H_0+V)^{m}-\sum_{j=0}^{p-1}\sum_{\substack{k_0,k_1,\dots,k_j\geq
0\\ k_0+k_1+\dots+k_j=m-j}}H_0^{k_0} VH_0^{k_1}V\dots
VH_0^{k_j}\bigg]\\&=\sum_{k=0}^\infty
z^{-(k+p+1)}\tau\bigg[(H_0+V)^{k+p}-\sum_{j=0}^{p-1}\sum_{\substack{k_0,k_1,\dots,k_j\geq
0\\ k_0+k_1+\dots+k_j=k+p-j}}H_0^{k_0} VH_0^{k_1}V\dots
VH_0^{k_j}\bigg].
\end{align}
By comparing the representations for $(-1)^p G_{\nu_p}^{(p)}(z)$ of
\eqref{f-la:gp1} and \eqref{f-la:gp4'}, we obtain that for any $k\in\{0\}\cup\Nats$,
\begin{align*}&\tau\bigg[(H_0+V)^{k+p}-\sum_{j=0}^{p-1}\sum_{\substack{k_0,k_1,\dots,k_j\geq
0\\ k_0+k_1+\dots+k_j=k+p-j}}H_0^{k_0} VH_0^{k_1}V\dots
VH_0^{k_j}\bigg]\\&\quad=(k+1)(k+2)\dots (k+p)\int_\Reals t^k\,d\nu_p(t),
\end{align*} along with Lemma \ref{prop:res'} proving the trace formula \eqref{f-la:tr} for all polynomials. We note that under the assumptions of Step 1, $p$ can be any natural number.

{\it Step 2.} Assume that $f\in\Rfr_b$, with $H_0$ not necessarily bounded.
It is enough to prove the statement for $f(t)=\frac{1}{(z-t)^{k+1}}$, $k\in\{0\}\cup\Nats$.
Applying Lemma \ref{prop:R_rec} gives
\begin{align}\label{f-la:z-1}
&p!\int_\Reals\frac{1}{(z-t)^{p+1}}\,d\nu_{p}(t)\\\nonumber
&\quad=(p-1)!\int_\Reals\frac{1}{(z-t)^{p}}\,d\nu_{p-1}(t)
-\tau\big[\big((zI-H_0)^{-1}V\big)^{p-1}(zI-H_0)^{-1}\big].
\end{align}
Differentiating \eqref{f-la:z-1} $k$ times with respect to $z$ gives
\begin{align}\label{f-la:z-2}
&(-1)^k(p+k)!\int_\Reals\frac{1}{(z-t)^{p+1+k}}\,d\nu_{p}(t)\\\nonumber
&=(-1)^k(p-1+k)!\int_\Reals\frac{1}{(z-t)^{p+k}}\,d\nu_{p-1}(t)
-\frac{d^k}{dz^k}\tau\big[\big((zI-H_0)^{-1}V\big)^{p-1}(zI-H_0)^{-1}\big].
\end{align}
Dividing by $(-1)^k k!$ on both sides of \eqref{f-la:z-2} implies
\begin{align}\label{f-la:z-3}
&\frac{(p+k)!}{k!}\int_\Reals\frac{1}{(z-t)^{p+1+k}}\,d\nu_{p}(t)\\\nonumber
&=\frac{(p-1+k)!}{k!}\int_\Reals\frac{1}{(z-t)^{p+k}}\,d\nu_{p-1}(t)
-\frac{(-1)^k}{k!}\frac{d^k}{dz^k}\tau\big[\big((zI-H_0)^{-1}V\big)^{p-1}(zI-H_0)^{-1}\big].
\end{align}
Making use of the representation
\[R_{p-1,H_0,V}\left(\frac{1}{(z-t)^{k+1}}\right)=
\frac{(p-1+k)!}{k!}\int_\Reals\frac{1}{(z-t)^{p+k}}\,d\nu_{p-1}(t)\] (see Theorem \ref{prop:tr'} for Koplienko's spectral shift function) and Lemma
\ref{prop:d=} converts \eqref{f-la:z-3} to
\begin{align}\label{f-la:z-4}
&\frac{(p+k)!}{k!}\int_\Reals\frac{1}{(z-t)^{p+1+k}}\,d\nu_{p}(t)\\\nonumber
&\quad=R_{p-1,H_0,V}\left(\frac{1}{(z-t)^{k+1}}\right)
-\frac12\tau\left[\frac{d^2}{dx^2}\bigg|_{x=0}\bigg((zI-H_0-xV)^{-k-1}\bigg)\right].
\end{align}By \eqref{f-la:rem},
\begin{align}\label{f-la:(z-1)^k+1}
&R_{p,H_0,V}\left(\frac{1}{(z-t)^{k+1}}\right)\\\nonumber
&\quad=R_{p-1,H_0,V}\left(\frac{1}{(z-t)^{k+1}}\right)
-\frac12\tau\left[\frac{d^2}{dx^2}\bigg|_{x=0}\bigg((zI-H_0-xV)^{-k-1}\bigg)\right].\end{align} Comparing \eqref{f-la:z-4} and \eqref{f-la:(z-1)^k+1}
completes the proof of \eqref{f-la:tr} for $f(t)=\frac{1}{(z-t)^{k+1}}$.
\end{proof}

\begin{proof}[Proof of Theorem \ref{prop:tr'} for $p=3$.] When $H_0$ is bounded,
Lemma \ref{prop:G_extend} and Theorem \ref{prop:more} prove the theorem for $f\in\Rfr$. Repeating the argument of Step 2 from the proof of Theorem \ref{prop:tr} (i) for $\tau$ the standard trace extends (ii) and (iii) of Theorem \ref{prop:tr'} to $f\in\mathcal{W}_p$ for $H_0$ bounded. Repeating the argument of Step 3 from the proof of Theorem \ref{prop:tr} (i) on each segment of $\Reals$ extends (i) to $f\in C_c^\infty(\Reals)$ for $H_0$ unbounded.
\end{proof}

\begin{proof}[Proof of Theorem \ref{prop:tr_free}]
(i) Due to Theorem \ref{prop:m_free}, there exists a bounded measure $\nu_p$ satisfying the trace formula \eqref{f-la:tr} for $f\in\mathcal{W}_p$. The proof repeats the proof of Theorem \ref{prop:tr} for the standard trace.

(ii) Using the moment--cumulant formula (see \cite[Theorem~2.17]{Speicher}),
we have
\begin{align}\label{f-la:free*}
\tau\big[\big((zI-H_0)^{-1}V\big)^{p-1}\big]=
\sum_{\pi=\{B_1,\ldots,B_\ell\}\in\operatorname{NC}(p-1)}k_{K(\pi)}[V,\ldots,V]
\prod_{j=1}^\ell\tau\big[(zI-H_0)^{-|B_j|}\big],
\end{align}
where (see the proof of Theorem~\ref{prop:m_free} for a bit of explanation, or \cite[Theorem~2.17]{Speicher} for a thorough description)
$k_{K(\pi)}[V,\ldots,V]$ is a polynomial of $\tau(V),\tau(V^2),\ldots,\tau(V^{p-1})$.
Since for $b\ge1$,
\begin{align*}
\tau\big[(zI-H_0)^{-b}\big]=\int_{\Reals^b}\frac{1}{(z-\la_1)\dots(z-\la_b)}\,
\tau\big(E_{H_0}(d\la_1)\cdots E_{H_0}(d\la_1)\big),
\end{align*}
we have
\[
\prod_{j=1}^\ell\tau\big[(zI-H_0)^{-b}\big]=\int_{\Reals^{p-1}}\frac1{(z-\la_1)\cdots(z-\la_p)}d\gamma_{p-1,\pi}(\la_1,\ldots,\la_p),
\]
where $\gamma_{p-1,\pi}$ is the measure described at~\eqref{f-la:mfree*1}.
Combining \eqref{f-la:free*} and \eqref{f-la:mfree*} gives
\[\tau\big[\big((zI-H_0)^{-1}V\big)^{p-1}\big]=
\int_{\Reals^{p-1}}\dd{\la_{p-1}}{p-2}{\frac1{z-\la}}\,dm_{p-1,H_0,V}(\la_1,\dots,\la_{p-1}).\]
Following the lines in the proof of Theorem \ref{prop:more} completes the proof of the absolute continuity of $\nu_p$ and repeating the proof of Lemma \ref{prop:G_extend}, Step 1, proves \eqref{f-la:tr} for $f$ a polynomial.
\end{proof}

\section{Spectral shift functions via basic splines}
\label{sec:xi}

We represent the density of the measure $\nu_p$ provided by Theorem \ref{prop:tr}
as an integral of a basic spline against a certain multiple spectral measure
when $H_0$ and $V$ are matrices. In addition, we show that existence of Krein's spectral shift function can be derived from the representation of the Cauchy transform via basic splines when
$\Mcal$ is finite. The representation of the Cauchy transform via basic splines, in its turn, follows from the double integral representation of $f(H_0+V)-f(H_0)$.


\begin{lemma}\label{prop:g} Let $\dim(\Hcal)<\infty$ and $H_0=H_0^*, V=V^*\in\Mcal=\BH$.
Then the Cauchy transform of the measure $\nu_p$ satisfying \eqref{f-la:tr} equals
\begin{align*}
G_{\nu_p}^{(p)}(z)=\frac{d^p}{dz^p}\bigg[(-1)^p\int_{\Reals^{p+1}}&
\dd{\la_{p+1}}{p}{\frac{1}{(p-1)!}(z-\la)^{p-1}\log(z-\la)}\\
&dm_{p,H_0,V}^{(2)}(\la_1,\la_2,\dots,\la_{p+1})\bigg],\quad\im(z)\neq 0.
\end{align*}
\end{lemma}

\begin{proof}Upon applying Remark \ref{prop:imr} and Lemma \ref{prop:g_der}, we obtain
\begin{align*}
G_{\nu_p}^{(p)}(z)=(-1)^p\int_{\Reals^{p+1}}\dd{\la_{p+1}}{p}{\frac{1}{z-\la}}\,
dm_{p,H_0,V}^{(2)}(\la_1,\la_2,\dots,\la_{p+1}).
\end{align*}
By Lemma \ref{prop:dd_der}, one of the antiderivatives of order $p$
of the function \[z\mapsto \dd{\la_{p+1}}{p}{\frac{1}{z-\la}}\]
equals
\[\dd{\la_{p+1}}{p}{\frac{1}{(p-1)!}(z-\la)^{p-1}\log(z-\la)-c_{p-1}(z-\la)^{p-1}},\]
where $c_{p-1}$ is a constant. Applying Proposition \ref{prop:dds}
\eqref{f-la:dd2} gives \[\dd{\la_{p+1}}{p}{c_{p-1}(z-\la)^{p-1}}=0,\]
completing the proof of the lemma.
\end{proof}

\begin{lemma}\label{prop:spline1}Let
$D_{p+1}=\{(\la_1,\la_2,\dots,\la_{p+1})\st\la_1=\la_2=\cdots=\la_{p+1}\in\Reals\}$.
Then, for any $(\la_1,\la_2,\dots,\la_{p+1})\in\Reals^{p+1}\setminus
D_{p+1}$ and $z\in\Complex\setminus\Reals$,
\begin{align*}
&\dd{\la_{p+1}}{p}{\frac{1}{(p-1)!}(z-\la)^{p-1}\log(z-\la)}\\&\quad
=\frac{1}{(p-1)!}\int_\Reals \frac{(-1)^p}{z-t}
\dd{\la_{p+1}}{p}{(\la-t)_+^{p-1}}\,dt.
\end{align*}
\end{lemma}

\begin{proof}
By Proposition \ref{prop:dds} \eqref{f-la:dd4},
\begin{align*}
&\dd{\la_{p+1}}{p}{\frac{1}{(p-1)!}(z-\la)^{p-1}\log(z-\la)}\\
&\quad=\frac{1}{(p-1)!}\int_\Reals\frac{\partial^p}{\partial t^p}
\left(\frac{1}{(p-1)!}(z-t)^{p-1}\log(z-t)\right)
\dd{\la_{p+1}}{p}{(\la-t)_+^{p-1}}\,dt\\
&\quad=\frac{1}{(p-1)!}\int_\Reals \frac{(-1)^p}{z-t}
\dd{\la_{p+1}}{p}{(\la-t)_+^{p-1}}\,dt.
\end{align*}
\end{proof}

\begin{thm}\label{prop:g_spline}Let $\dim(\Hcal)<\infty$ and $H_0=H_0^*, V=V^*\in\Mcal=\BH$. Then the Cauchy transform of the measure $\nu_p$ satisfying \eqref{f-la:tr} equals
\begin{align*}
&G_{\nu_p}(z)\\&=\int_\Reals\frac{1}{z-t}
\left(\frac{1}{(p-1)!}\int_{\Reals^{p+1}\setminus_{D_{p+1}}}
\dd{\la_{p+1}}{p}{(\la-t)_+^{p-1}}
dm_{p,H_0,V}^{(2)}(\la_1,\la_2,\dots,\la_{p+1})\right)\,dt.
\end{align*}
\end{thm}

\begin{proof}By Lemmas \ref{prop:g} and \ref{prop:spline1},
\begin{align}\label{f-la:spline2}\nonumber
&G_{\nu_p}(z)=\text{\rm
pol}_p(z)\\\nonumber&+\frac{1}{(p-1)!}\int_{\Reals^{p+1}\setminus_{D_{p+1}}}
\left(\int_\Reals\frac{1}{z-t}
\dd{\la_{p+1}}{p}{(\la-t)_+^{p-1}}\,dt\right)
dm_{p,H_0,V}^{(2)}(\la_1,\la_2,\dots,\la_{p+1})\\
&+\frac{1}{(p-1)!}\int_{D_{p+1}}\frac{1}{z-\la}\,dm_{p,H_0,V}^{(2)}(\la,\la,\dots,\la),
\end{align} where $\text{\rm pol}_p(z)$ is a polynomial of degree $\leq
p$. As stated in Proposition \ref{prop:dds} \eqref{f-la:dd5}, the
basic spline $\dd{\la_{p+1}}{p}{(\la-t)_+^{p-1}}$ is non-negative and
integrable, with the $L_1$-norm equal to $1/p$. By Corollary
\ref{prop:m.k}, the measure $m_{p,H_0,V}^{(2)}$ has bounded variation. On
one hand, it guarantees that the first integral in
\eqref{f-la:spline2} is $\mathcal{O}(1/\im(z))$ as $\im(z)\rightarrow
+\infty$. On the other hand, it allows to change the order of
integration in the first integral in \eqref{f-la:spline2}. By Lemma
\ref{prop:no_atoms}, the second integral in \eqref{f-la:spline2} equals $0$. Comparing the asymptotics of $G_{\nu_p}(z)$ and the
integrals in \eqref{f-la:spline2} as $\im(z)\rightarrow +\infty$
implies that $\text{\rm pol}_p(z)=0$, completing the proof of the
theorem.
\end{proof}

\begin{cor}Let $\dim(\Hcal)<\infty$ and $H_0=H_0^*, V=V^*\in\Mcal=\BH$.
Then the density of the measure $\nu_p$ satisfying \eqref{f-la:tr} equals
\[\eta_p(t)=\frac{1}{(p-1)!}\int_{\Reals^{p+1}\setminus_{D_{p+1}}}
\dd{\la_{p+1}}{p}{(\la-t)_+^{p-1}}
dm_{p,H_0,V}^{(2)}(\la_1,\la_2,\dots,\la_{p+1}),\] for a.e.
$t\in\Reals$.
\end{cor}

\begin{proof}
By Theorem \ref{prop:g_spline}, the Cauchy transforms of $\nu_p$ and $\eta_p(t)dt$ coincide.
This implies (see Step 1 in the proof of Lemma \ref{prop:G_extend}) that the functionals given by $\nu_p$ and $\eta_p(t)dt$ coincide on the polynomials defined on $[a,b]$, where $[a,b]$ contains the spectra of $H_0$ and $H_0+V$. Hence, $d\nu_p=\eta_p(t)dt$.
\end{proof}

Below, we prove absolute continuity of $\nu_1$ by techniques different from those of \cite{Krein1}.

\begin{thm}Let $\tau$ be finite. Let $H_0=H_0^*$ be an operator affiliated with $\Mcal$ and $V=V^*\in\Mcal$. The trace formula \eqref{f-la:tr} with $p=1$ holds for every $f\in\mathcal{W}_1$, with $\nu_1$ absolutely continuous. The density $\eta_1$ of $\nu_1$ is given by the formula
\begin{align}
\label{f-la:new_eta1}
\eta_1(t)=\int_{\Reals^2\setminus_{D_2}}
\frac{1}{|\mu-\la|}\chi_{(\min\{\la,\mu\},\max\{\la,\mu\})}(t)\,dm_{1,H_0,V}^{(2)}(\la,\mu),
\end{align} for a.e. $t\in\Reals$.
If, in addition, $H_0$ is bounded, then \eqref{f-la:tr} holds for $f\in\Rfr$.
\end{thm}

\begin{proof}
Repeating the argument in the proof of Theorem \ref{prop:g_spline} leads to
the formula \eqref{f-la:spline2}. By Lemma \ref{prop:R_spline}, the measure $m_{1,H_0,V}^{(2)}$ is real-valued. Then by Lemma \ref{prop:no_atoms} and the Poisson inversion, for any $x\in\Reals$,
\[\lim_{\varepsilon\rightarrow 0^+}
\im\left(\int_{D_{2}}\frac{1}{x+\i\varepsilon-\la}\,dm_{1,H_0,V}^{(2)}(\la,\la)\right)=0,\]
proving Krein's trace formula for $f=f_z$ with $\eta_1$ given by \eqref{f-la:new_eta1}.
Adjusting the argument in the proof of Lemma \ref{prop:G_extend}, Step 2, extends \eqref{f-la:tr} to $f\in\Rfr_b$. Repeating the argument in the proof of
\cite[Lemma 8.3.2]{Yafaev} extends the result of the theorem from $f\in\Rfr$ to $f\in\mathcal{W}_1$ 
with the same absolutely continuous measure $d\nu_1(t)=\eta_1(t)\,dt$.
\end{proof}

\section{Higher order spectral averaging formulas}
\label{sec:sa}

\begin{thm}\label{prop:hosa}
Assume that $H_0=H_0^*\in\Mcal$ and either $\tau$ is standard or $p=2$. Let $V\in\ncs{2}$. Then the measure
\begin{align*}
\int_0^1(1-x)^{p-1}\tau\big[(E_{H_0+xV}(dt)V)^p\big]\,dx
\end{align*}is absolutely continuous with the density equal to
\begin{align*}
&\eta_p(t)(p-1)!\\&-p\int_0^1(1-x)^{p-1}\int_{\Reals^{p+1}\setminus D_{p+1}}\dd{\la_{p+1}}{p}{(\la-t)^{p-1}_+}
\,dm_{p,H_0+xV,V}^{(1)}(\la_1,\dots,\la_{p+1})\,dx.
\end{align*}
\end{thm}

\begin{proof}Let $[a,b]\supset\sigma(H_0)\cup\sigma(H_0+V)$. Then
by Theorem \ref{prop:Rpf} \eqref{f-la:Rpf} and Remark \ref{prop:imrp},
\begin{align*}
&\tau[R_{p,H_0,V}(f)]\\&\quad=\frac{1}{(p-1)!}\int_0^1(1-x)^{p-1}
\tau\left[\frac{d^p}{dx^p}f(H_0+xV)\right]\,dx\\
&\quad=\frac{1}{(p-1)!}\int_0^1(1-x)^{p-1}
p!\int_{\Reals^{p+1}}\dd{\la_{p+1}}{p}{f}dm_{p,H_0+xV,V}^{(1)}(\la_1,\dots,\la_{p+1})\,dx,
\end{align*} for $f\in C_c^\infty(\Reals)$ such that $f\big|_{[a,b]}$ coincides with a polynomial. Applying Proposition \ref{prop:dds} \eqref{f-la:dd4} and then changing the order of integration yield
\begin{align}\label{f-la:sa1}\nonumber
&\tau[R_{p,H_0,V}(f)]\\\nonumber
&\quad=\frac{p}{(p-1)!}\int_0^1(1-x)^{p-1}\int_{\Reals^{p+1}\setminus
D_{p+1}}\bigg(\int_\Reals f^{(p)}(t)\\\nonumber
&\quad\quad\quad\quad\quad\quad\quad\quad\quad\dd{\la_{p+1}}{p}{(\la-t)^{p-1}_+}
\,dt\bigg)\,dm_{p,H_0+xV,V}^{(1)}(\la_1,\dots,\la_{p+1})\,dx\\\nonumber
&\quad\quad+\frac{1}{(p-1)!}\int_0^1(1-x)^{p-1}\int_{D_{p+1}}
f^{(p)}(\la)\,dm_{p,H_0+xV,V}^{(1)}(\la,\dots,\la)\,dx.\\
&\quad=\int_\Reals
f^{(p)}(t)\frac{p}{(p-1)!}\bigg(\int_0^1(1-x)^{p-1}\\\nonumber
&\quad\quad\quad\quad\quad\quad\int_{\Reals^{p+1}\setminus
D_{p+1}}\dd{\la_{p+1}}{p}{(\la-t)^{p-1}_+}
\,dm_{p,H_0+xV,V}^{(1)}(\la_1,\dots,\la_{p+1})\,dx\bigg)\,dt\\\nonumber
&\quad\quad+\int_\Reals f^{(p)}(t)\frac{1}{(p-1)!}\int_0^1(1-x)^{p-1}\tau\big[(E_{H_0+xV}(dt)V)^p\big]\,dx.
\end{align}
Along with Theorem \ref{prop:tr} in the case of $\Mcal=\BH$ or Theorem \ref{prop:tr'} in the case of a general $\Mcal$, respectively, \eqref{f-la:sa1} implies that
\begin{align}\label{f-la:sa3}\nonumber
&\int_\Reals f^{(p)}(t)\frac{1}{(p-1)!}\int_0^1(1-x)^{p-1}\tau\big[(E_{H_0+xV}(dt)V)^p\big]\,dx\\
&\quad=\int_\Reals f^{(p)}(t)\eta_p(t)\,dt-\int_\Reals
f^{(p)}(t)\frac{p}{(p-1)!}\bigg(\int_0^1(1-x)^{p-1}\\\nonumber
&\quad\quad\quad\quad\quad\quad\int_{\Reals^{p+1}\setminus
D_{p+1}}\dd{\la_{p+1}}{p}{(\la-t)^{p-1}_+}
\,dm_{p,H_0+xV,V}^{(1)}(\la_1,\dots,\la_{p+1})\,dx\bigg)\,dt,
\end{align}from which the statement of the theorem follows.
\end{proof}

\begin{remark}
The assertion of Theorem \ref{prop:hosa} remains true if $\tau$ is finite, $H_0$ and $V$ are free in $(\Mcal,\tau)$, and $p\geq 2$.
\end{remark}

\begin{remark}The argument in the proof of Theorem \ref{prop:hosa} can be repeated for $p=1$, provided $V\in\ncs{1}$.
Since $m_{1,H_0+xV,V}^{(1)}(A_1\times A_2)=\tau[E_{H_0+xV}(A_1\cap A_2)V]$, for $A_1,A_2\in\Reals$, one has that $m_{1,H_0+xV,V}^{(1)}(\Reals^{p+1}\setminus D_{p+1})=0$. Therefore, \eqref{f-la:sa1} converts to
\begin{align*}
&\tau[f(H_0+V)-f(H_0)]\\
&\quad=\int_0^1\int_\Reals
f'(t)\tau[E_{H_0+xV}(dt)V]\,dx=\int_\Reals
f'(t)\int_0^1\tau[E_{H_0+xV}(dt)V]\,dx.
\end{align*} Along with Krein's trace formula the latter implies that
$\int_0^1\tau[E_{H_0+xV}(dt)V]\,dx=\eta_1(t)\,dt$.
\end{remark}

\subsection*{Acknowledgment}
The authors thank Bojan Popov for helpful conversations about splines.

\bibliographystyle{plain}

\begin{thebibliography}{99}

\bibitem{Azamov0} N.A. Azamov, A.L. Carey, P.G. Dodds, F.A. Sukochev,
Operator integrals, spectral shift, and spectral flow, \emph{Canad.
J. Math.,} to appear.

\bibitem{Azamov} N.A. Azamov, P.G. Dodds, F.A. Sukochev,
The Krein spectral shift function in semifinite von Neumann algebras,
\emph{Integr. Equ. Oper. Theory} {\bf 55} (2006), 347 -- 362.

\bibitem{BirmanLN}
M.Sh. Birman, Spectral shift function and double operator integrals, Linear and complex analysis problem book 3, 1573, Eds. V.P. Havin et al,
Lecture Notes in Math., Springer-Verlag, Berlin-Heidelberg-New York, (1994), pp. 272--273.

\bibitem{Birman72}
M.Sh. Birman, M.Z. Solomyak, Remarks on the
spectral shift function, \emph{Zapiski Nauchn. Semin. LOMI}  {\bf 27}
(1972), 33--46 (Russian). Translation: \emph{J. Soviet Math.} {\bf 3}
(1975), 408--419.

\bibitem{Birman96}
M.Sh.Birman, M.Z.Solomyak, Tensor product of a finite number of
spectral measures is always a spectral measure, \emph{Integr. Equ.
Oper. Theory} {\bf 24} (1996), 179--187.

\bibitem{BoucletKo}
J.M. Bouclet, Trace formulae for relatively Hilbert-Schmidt
perturbations, \emph{Asymptot. Anal.} {\bf 32} (2002), 257 -- 291.

\bibitem{Carey}
R.W. Carey, J.D. Pincus, Mosaics, principal functions, and mean
motion in von Neumann algebras, \emph{Acta Math.} {\bf 138} (1977),
153 -- 218.

\bibitem{Daletskii}
Yu.L. Daletskii, S.G. Krein, Integration and differentiation of
functions of Hermitian operators and application to the theory
perturbations, \emph{Trudy Sem. Funktsion. Anal., Voronezh Gos.
Univ.} {\bf 1} (1956), pp. 81 -- 105. (Russian)

\bibitem{Vore}
R.A. DeVore, G.G. Lorentz, Constructive
Approximation, Grundlehren der Mathematischen Wissenschaften, 303,
Springer-Verlag, Berlin, 1993.

\bibitem{Dostanic}
M. Dostani\`{c}, Trace formula for nonnuclear perturbations of
selfadjoint operators, \emph{Publ. Inst. Math\`{e}matic} {\bf 54}
(68) (1993), 71--79.

\bibitem{MakarovL}
F. Gesztesy, K.A. Makarov, S.N. Naboko, The spectral shift
operator, in J.~Dittrich, P.~Exner, and M.~Tater (eds.)
``Mathematical Results in Quantum Mechanics", \emph{Operator Theory:
Advances and Applications,} {\bf 108}, Birkh\"{a}user, Basel, 1999,
pp. 59--90.

\bibitem{GPS} F. Gesztesy, A. Pushnitski, B. Simon, On the Koplienko
spectral shift function, I. Basics, \emph{Zh. Mat. Fiz. Anal. Geom.}
{\bf 4} (2008), no.~1, 63 -- 107.

\bibitem{HS} U. Haagerup, H. Schultz,
Invariant Subspaces for Operators in a General $II_1$-factor,
preprint, arXiv:math/0611256.

\bibitem{bimeasure}
S. Karni, E. Mezbach, On the extension of bimeasures, \emph{J. Analyse Math.}
{\bf 55} (1990), 1--16.

\bibitem{Kop84} L.S. Koplienko,
Trace formula for perturbations of nonnuclear type, \emph{Sibirsk.
Mat. Zh.}  {\bf 25}  (1984), 62-71 (Russian). Translation: Trace
formula for nontrace--class perturbations, \emph{Siberian Math. J.}
{\bf 25} (1984), 735--743.

\bibitem{Krein1}
M.G. Krein, On a trace formula in perturbation theory, \emph{Matem.
Sbornik} {\bf 33} (1953), 597 -- 626 (Russian).

\bibitem{Lifshits}
I.M. Lifshits, On a problem of the theory of perturbations connected
with quantum statistics, \emph{Uspehi Matem. Nauk} {\bf 7} (1952),
171 -- 180 (Russian).

\bibitem{ms}
K.A. Makarov, A. Skripka, Some applications of the perturbation
determinant in finite von Neumann algebras, \emph{Canad. J. Math.,}
to appear.


\bibitem{Pagter}
B. de Pagter, F.A. Sukochev, Differentiation of operator functions in non-commutative $L^p$-spaces, \emph{J. Funct. Anal.} {\bf 212} (2004), no.~1, 28--75.

\bibitem{Pavlov}
B.S. Pavlov, On multidimensional operator integrals, \emph{Problems
of mathematial analysis, Leningrad State Univ.} {\bf 2} (1969),
99--121 (Russian).

\bibitem{PellerKo}
V.V. Peller, An extension of the Koplienko-Neidhardt trace formulae,
\emph{J. Funct. Anal.} {\bf 221} (2005), 456--481.

\bibitem{PellerKr}
V.V. Peller, Hankel operators in the perturbation theory of unbounded
self-adjoint operators. Analysis and partial differential
equations, Lecture Notes in Pure and Applied
Mathematics, 122, Dekker, New York, 1990, pp. 529--544.

\bibitem{PellerMult}
V.V. Peller, Multiple operator integrals and higher operator
derivatives, \emph{J. Funct. Anal.} {\bf 223} (2006), 515--544.

\bibitem{Schwartz}
J.T. Schwartz, Nonlinear functional analysis, Gordon and Breach Science Publishers, New York - London - Paris, 1969.

\bibitem{Simonbs}
B.~Simon, Spectral averaging and the Krein spectral shift,
\emph{Proc. Amer. Math. Soc.} {\bf 126} (1998), no.~5, 1409--1413.

\bibitem{convexity}
A. Skripka, Trace inequalities and spectral shift, \emph{Oper. Matrices}, to appear.

\bibitem{Speicher}
R.~Speicher, Free Calculus,
\emph{Quantum probability communications, Vol. XII (Grenoble, 1998)},
World Sci.\ Publ., River Edge, NJ, 2003, pp. 209--235.

\bibitem{VDN} D.V.\ Voiculescu, K.J.~Dykema, A.~Nica,
{\em Free Random Variables},
CRM Monograph Series {\bf 1}, American Mathematical Society, 1992.

\bibitem{Yafaev}
D.R. Yafaev, Mathematical scattering theory: general theory,
Providence, R.I., AMS, 1992.
\end{thebibliography}

\end{document}